\documentclass[12pt,a4paper]{article}
\usepackage[utf8]{inputenc}
\usepackage[english]{babel}
\usepackage[a4paper,margin=1in]{geometry}

\usepackage[font=footnotesize,labelfont=bf]{caption}
\usepackage{amssymb}
\usepackage{amsfonts}
\usepackage{amsmath}
\usepackage{amsthm}

\usepackage{sectsty}
\sectionfont{\large}
\subsectionfont{\normalsize}
\subsubsectionfont{\small}

\usepackage{epsfig}
\usepackage{graphicx}
\usepackage{pgfplots}
\usepackage{lmodern}
\usepackage{dsfont}
\usepackage{mathrsfs}
\usepackage{latexsym}

\usepackage{verbatim}
\usepackage{setspace}
\usepackage{enumitem}
\usepackage{mathtools}

\renewcommand{\labelenumi}{\textup{(\roman{enumi})}}
\renewcommand\theenumi\labelenumi 
\renewcommand{\geq}{\geqslant}
\renewcommand{\leq}{\leqslant}

\theoremstyle{plain}
\newtheorem{theorem}{Theorem}[section]
\newtheorem{corollary}[theorem]{Corollary}
\newtheorem{lemma}[theorem]{Lemma}

\theoremstyle{definition}

\newtheorem{remark}[theorem]{Remark}

\newtheorem{example}[theorem]{Example}

\newcommand{\domain}{\mathscr{D}(\cA)}
\newcommand\diff{\mathrm{d}}
\newcommand{\real}{\mathds{R}}

\newcommand{\nat}{\mathds{N}}

\newcommand{\X}{(X_t)_{t\geq0}}

\newcommand{\Ym}{(\mathcal{Y}_t)_{t\geq0}}
\newcommand{\Xmm}{(\mathcal{X}^{(m)}_t)_{t\geq0}}
\newcommand{\Ymm}{(\mathcal{Y}^{(m)}_t)_{t\geq0}}
\newcommand{\Y}{(Y_t)_{t\geq0}}
\renewcommand{\L}{(L_t)_{t\geq0}}
\newcommand{\Ln}{(L^{(n)}_t)_{t\geq0}}

\newcommand{\testo}{\cC_c^\infty(\real)}
\newcommand{\integer}{\mathds{Z}}

\newcommand{\Xn}{(X^{(n)}_t)_{t\geq0}}
\newcommand{\B}{(B_t)_{t\geq0}}
\newcommand{\N}{(N_t)_{t\geq0}}
\newcommand{\sobolev}{W^{1,1}_\loc(0,\infty)}
\newcommand{\ppi}{\boldsymbol{\pi}}
\newcommand{\eeta}{\boldsymbol{\eta}}

\newcommand{\PP}{\mathds{P}}

\newcommand{\EE}{\mathds{E}}
\newcommand{\cB}{\mathcal{B}}
\newcommand{\cC}{\mathcal{C}}

\newcommand{\cD}{\mathcal{D}}

\newcommand{\cF}{\mathcal{F}}
\newcommand{\cG}{\mathcal{G}}
\newcommand{\cP}{\mathcal{P}}
\newcommand{\cA}{\mathcal{A}}

\newcommand{\cE}{\mathcal{E}}
\newcommand{\cO}{\mathcal{O}}

\newcommand{\ee}{\mathrm{e}}
\newcommand{\bone}{\mathds{1}}

\newcommand{\supp}{\operatorname{supp}}

\newcommand{\sgn}{\operatorname{sgn}}

\newcommand{\loc}{\mathrm{loc}}

\renewcommand{\epsilon}{\varepsilon} 

\usepackage{xcolor}

\definecolor{see}{HTML}{AA7700}
\definecolor{doubt}{HTML}{AF0000}
\definecolor{todo}{HTML}{0000AA}

\usepackage{marginnote}
\begin{document}
\frenchspacing
\allowdisplaybreaks[4]

\title{\bfseries Lévy Langevin Monte Carlo}
\author{David Oechsler\thanks{Technische Universit\"at
		Dresden, Institut f\"ur Mathematische Stochastik, Helmholtzstra{\ss}e 10, 01069 Dresden, Germany. \texttt{david.oechsler@tu-dresden.de}}~\thanks{Center of Scalable Data Analytics and Artificial Intelligence (ScDS.AI) Dresden/Leipzig, Germany}\;  }
\date{\today}
\maketitle
\begin{abstract}\noindent
Analogue to the well-known Langevin Monte Carlo method, in this article we provide a method to sample from a target distribution \(\ppi\) by simulating a solution of a stochastic differential equation. Hereby, the stochastic differential equation is driven by a general Lévy process which - other than in the case of Langevin Monte Carlo - allows for non-smooth targets. Our method will be fully explored in the particular setting of target distributions supported on the half-line \((0,\infty)\) and a compound Poisson driving noise. Several illustrative examples conclude the article.


\end{abstract}

\section{Introduction}

Monte Carlo methods based on stationary Markov processes appear frequently in fields such as statistics, computer simulation and machine learning, and they have a variety of applications, for example in physics and biology, cf. \cite{bardenet2017markov,brooks2011handbook,kendall2005markov,FLI,welling2011bayesian}. These methods have in common that in order to sample from a target distribution \(\ppi\) one considers sample paths of certain Markov processes to approximate \(\ppi\).\\
\emph{Langevin Monte Carlo} (\textsc{lmc}) is one of these methods and it originates from statistical physics. It applies to absolutely continuous target distributions \(\ppi(\diff x)=\pi(x)\diff x\) with smooth density functions \(\pi:\real^d\to\real_+\), and its associated process \(\X\) is the so-called \emph{Langevin diffusion}, that is a strong solution of the stochastic differential equation (\textsc{sde})
\begin{align}\label{eq_LMC}
\diff X_t = -\frac{\nabla \pi(X_t)}{\pi(X_t)}\diff t + \sqrt{2}\diff B_t,
\end{align}
where \(\B\) is a standard Brownian motion on \(\real^d\), and \(\nabla\pi\) denotes the gradient of \(\pi\). For \textsc{LMC} to produce samples from \(\ppi\) it is required that \(\X\) is a unique strong solution for \eqref{eq_LMC} and \(\ppi\) is an invariant distribution for \(\X\), that is
\begin{align}\label{eq_invmeas}
\int_\real\PP^x(X_t\in B)\ppi(\diff x)=\ppi(B)\quad \text{for all}\quad t\geq0, B\in\cB(\real).
\end{align}
However, for this to be the case it is only natural that assumptions must be made regarding \(\ppi\), e.g. that \(\nabla \pi\) exists in a suitable sense. Moreover, to sample from \(\ppi\) using \(\X\) it is essential that \(\X\) converges to \(\ppi\) in a suitable sense from any starting point \(x\in\supp\ppi\). As solutions \(\X\) of \eqref{eq_LMC} are almost surely continuous, \(\supp\ppi\) is necessarily connected for convergence to be even possible. For more on \textsc{lmc} see \cite{brooks2011handbook} or \cite{roberts1996exponential}. \\
Due to the constraints of \textsc{lmc} it is reasonable to ask whether one could construct similar methods by replacing the Brownian motion with a more general process. In this article we consider Lévy processes as driving noises. In particular, we are interested in the following question:\\

\noindent
\textit{Given a distribution \(\ppi\) and a Lévy process \(\L\), can we choose a drift coefficient \(\phi\) such that we can sample from \(\ppi\) by simulation of a solution \(\X\) of 
\begin{align}\label{eq_sde}
\diff X_t &= \phi(X_{t})\diff t + \diff L_t?
\end{align}}

\noindent
There are various cases in the literature for which \textsc{sde}s of the form \eqref{eq_sde} are considered with Lévy processes as driving noises. In \cite{FLII} and \cite{FLI} a \emph{fractional Langevin Monte Carlo} (f\textsc{lmc}) method is introduced for which \(\L\) is an \(\alpha\)-stable process, and in \cite{eliazar}, several examples are produced for the case when the driving noise is a pure jump Lévy process. However, both of these studies are rather focused on practical aspects, disregarding some of the theoretical foundations. This will be discussed further in Remark \ref{rem_source}.\\
To thoroughly answer the above question it is essential to distinguish between the notions \emph{infinitesimally invariant distribution}, \emph{invariant distribution}, and \emph{limiting distribution}. These, and some more introductory notions and well-known facts can be found in Section \ref{sec_prel} of this article. After that, in Section \ref{sec_inva}, we investigate under which conditions a drift coefficient \(\phi\) exists such that \(\ppi\) is infinitesimally invariant for \(\X\). Clearly, there are cases for which this is not the case, think of discrete distributions, or distributions on a half-space while jumps can occur in all directions. Hence, a general answer can only exist under certain assumptions on the regularity of \(\ppi\) and the compatibility of \(\ppi\) and \(\L\). \\
In the same section, we then find a particular set of conditions under which \(\ppi\) is invariant and limiting for \(\X\). Various examples subsequently illustrate our results. Afterwards, in Section \ref{sec_prf}, we present the more technical aspects of the proofs, followed in Section \ref{sec_out} by a list of possible extensions with comments on the difficulties they might pose.\\
Methodologically we rely on the results in \cite{behmeoechsler} on invariant measures of Lévy-type processes, the Foster-Lyapunov methods originating in a series of articles by S.P. Meyn and R.L. Tweedie (\cite{MTI}-\cite{MTIII}), and standard techniques from the theory of ordinary differential equations.

\section{Preliminaries}\label{sec_prel}

Throughout this paper we denote by \(L^{p}(\real)\) and \(W^{k,p}(\real)\) the classic Lebesgue and Sobolev spaces, and by \(\cC_\infty(\real)\) the space of continuous functions vanishing at infinity, i.e. functions \(f\in\cC(\real)\) such that for all \(\varepsilon>0\) there exists a compact set \(K\subset\real\) such that for all \(x\in\real\setminus K\) it holds \(|f(x)|<\varepsilon\). \\
Let \(U\subseteq\real\) be an open set. As usual, \(\cC_c^\infty(U)\) denotes the space of test functions, i.e. smooth functions \(f\) with compact support \(\supp f\subset U\). Linear functionals \(T:\cC_c^\infty(U)\to\real\) that are continuous w.r.t. uniform convergence on compact subsets of all derivatives are called \emph{Schwartz distributions}. The \emph{distributional derivative} of a Schwartz distribution \(T\) is defined by \(T':\cC_c^\infty(U)\to\real, f\mapsto T(f')\). If there exists \(N\in \nat_0\) such that for all compact sets \(K\subset U\) there exists \(c>0\) such that 
\begin{align*}
|T(f)|\leq c\max\{|f^{(n)}|: n\leq N\}
\end{align*}
for all \(f\in\cC_c^\infty(U)\) with \(\supp f\subset K\), then the smallest such \(N\) is called the \emph{order} of \(T\). If no such \(N\) exists, the order of \(T\) is set to \(\infty\).

\paragraph{Markov processes and generators} 
Let \(\X\) be a Markov process in \(\real\) on the probability space \((\Omega,\cF,\PP)\). We denote
\begin{align*}
\PP^x(~\cdot~):=\PP(~\cdot~|X_0=x)\quad\text{and}\quad \EE^x[~\cdot~]:=\EE[~\cdot~|X_0=x]
\end{align*}
for all \(x\in\real\).
The \emph{pointwise generator} of \(\X\) is the pair \((\cA,\domain)\) defined by
\begin{align*}
\cA f(x):=\lim_{t\downarrow0} \frac{\EE^x f(X_t) - f(x)}{t},\quad x\in\real,\quad f\in\domain
\end{align*}
where 
\begin{align*}
\domain:=\left\{f\in\cC_\infty(\real): \lim_{t\downarrow0} \frac{\EE^x f(X_t) - f(x)}{t} \text{ exists for all } x\in\real \right\}.
\end{align*}
Further, denote by \(\cD(\cG)\) the set of all functions \(f:\real\to\real\) for which there exists a measurable function \(g:\real\to\real\) such that for all \(x\in\real\) and \(t>0\) it holds
\begin{align*}
\EE^xf(X_t) = f(x) + \EE^x\left[\int_0^t g(X_s)\diff s\right]
\end{align*}
and
\begin{align*}
\int_0^t\EE^x\left[ |g(X_s)|\right]\diff s <\infty.
\end{align*}
Setting \(\cG f=g\) the pair \((\cG,\cD(\cG))\) is called the \emph{extended generator} of \(\X\).

\paragraph{Lévy processes}
A (one-dimensional) Lévy process \(\L\) is a Markov process with stationary and independent increments with characteristic exponent \(\varphi(\beta):=\ln\EE[\ee^{i\beta L_1}]\) given by  
\begin{align*}
\varphi(\beta)=i\gamma\beta - \frac12\sigma^2\beta^2 + \int_{\real}\left(\ee^{i\beta z}-1\right)\mu(\diff z) +  \int_{\real^d}\left(\ee^{i\beta z}-1-i\beta z \bone_{\{|z|<1\}}\right)\rho(\diff z).
\end{align*}
Here, \(\gamma\in\real\) is the \emph{  location   parameter}, \(\sigma^2\geq0\) is the \emph{Gaussian parameter}, and \(\mu\) and \(\rho\) are two measures on \(\real\) such that \(\mu\{0\}=\nu\{0\}=0\) and \(\int_{\real}(1\wedge|z|)\mu(\diff z)<\infty\) and \(\int_{\real}(|z|\wedge|z|^2)\rho(\diff z)<\infty\), respectively. The measure \(\Pi=\mu+\rho\) is called the \emph{jump measure}. The triplet \((\gamma,\sigma^2,\Pi) =(\gamma,\sigma^2,\mu+\rho) \) is called the \emph{characteristic triplet} of \(\L\). Note that the decomposition of \(\Pi\) into \(\mu\) and \(\rho\) is not unique.\\
Further, denote by
\begin{align*}
\overline\mu(x)&=\begin{cases}
\mu(x,\infty), &x>0,\\
\mu(-\infty,x), &x<0,\\
0, & x=0,
\end{cases}
\end{align*}
the \emph{integrated tail} of \(\mu\), and by \(\overline\mu_s(x):=\sgn(x)\overline\mu(x)\) the \emph{signed integrated tail} of \(\mu\). We similarly define the \emph{double integrated tail} of \(\rho\) by
\begin{align*}
\overline{\overline\rho}(x)&=\begin{cases}
\int_{(x,\infty)}\rho(z,\infty)\diff z, &x>0,\\
\int_{(-\infty,x)}\rho(-\infty,z)\diff z, &x<0,\\
0, & x=0.
\end{cases}
\end{align*}
A Lévy process \(\L\) with \(\sigma^2=0,~\rho=0\) and \(|\mu|<\infty\) is called a \emph{compound Poisson process}. If, additionally, \(\supp\mu\subset\real_+\) then \(\L\) is called a \emph{spectrally positive compound Poisson process}.

\paragraph{Invariant measures and Harris recurrence} Let \(\X\) be a Markov process on \(\real\) with open state space \(\mathcal O\subseteq\real\) and with pointwise generator \((\cA,\domain)\). As mentioned in the introduction, a measure \(\ppi\) with \(\supp\ppi\subset\overline{\mathcal O}\) is called \emph{invariant} for \(\X\) if \eqref{eq_invmeas} holds. It is called \emph{infinitesimally invariant} for \(\X\) if
\begin{align*}
\int_\real \cA f(x)\ppi(\diff x) = 0\quad \text{for all } f\in\cC_c^\infty(\mathcal O),
\end{align*}
and \(\ppi\) is called \emph{limiting} for \(\X\) if \(\ppi\) is a \emph{distribution}, i.e. \(|\ppi|=1\), and
\begin{align*}
\lim_{t\to\infty}\|\PP^x(X_t\in \cdot)-\ppi\|_{\mathrm{TV}}=0, \quad \text{for all } x\in\mathcal O,
\end{align*}
where \(\|\cdot\|_{\mathrm{TV}}\) denotes the total variation norm.\\
The process \(\X\) is called \emph{Harris recurrent} if there exists a non-trivial \(\sigma\)-finite measure \(a\) on \(\real\) such that for all \(B\in\cB(\real)\) with \(a(B)>0\) it holds \(\PP^x(\tau_B<\infty)=1\) where \(\tau_B:=\inf\{t\geq0: X_t\in B\}\). It is well-known (cf. \cite{MTIII}) that for any Harris recurrent Markov process \(\X\) an invariant measure \(\ppi\) exists which is unique up to multiplication with a constant.  If \(\ppi\) is finite it can be normalized to be a distribution. In this case \(\X\) is called \emph{positive Harris recurrent}.

\section{Lévy Langevin Monte Carlo}\label{sec_inva}

Let \(\ppi\) be a probability distribution on \(\real\), and let \(\L\) be a Lévy process on \(\real\). Further, let \(\X\) be a solution of
\begin{align*}
\diff X_t = \phi(X_t)\diff t + \diff L_t.
\end{align*}
Can we choose \(\phi:\real\to\real\) in such a way that \(\ppi\) is limiting for \(\X\)? In the spirit of f\textsc{lmc} we call the sampling of \(\X\) in order to sample from \(\ppi\) \emph{Lévy Langevin Monte Carlo} (\textsc{llmc}).\\
As mentioned in the introduction a general answer to this question cannot be given without certain conditions on \(\ppi\) and \(\L\). Throughout, we assume that \(\ppi(\diff x)=\pi(x)\diff x\) is absolutely continuous, and that the Lévy process \(\L\) with characteristic triplet \((\gamma,\sigma^2,\Pi)\) is not purely deterministic, i.e. \(\sigma^2>0\) or \(\Pi\neq0\). Recall that \(\Pi=\mu+\rho\) with \(\mu\) and \(\rho\) as in Section \ref{sec_prel}.\\
Define \(\cE:=\{x\in\real: \pi(x)>0\}\) and assume either \(\cE=\real\) or some open half-line - without loss of generality we choose in this case \(\cE=(0,\infty)\). This choice of \(\cE\) is not a real restriction, as explained further in Section \ref{sec_out}. Additionally we assume the following:
\begin{enumerate}
\item[\textbf{(a1)}] If \(\cE=(0,\infty)\), then \(\L\) is a spectrally positive compound Poisson process.
\item[\textbf{(a2)}] If \(\cE=(0,\infty)\), then there exists \(c>0\) such that \(\int_0^x\pi(z)\diff z\leq cx\pi(x)\) for all \(x\ll1\).
\item[\textbf{(a3)}] If \(\L\) has paths of unbounded variation, then \(\pi\in\sobolev\).
\end{enumerate}

\subsection{Infinitesimally invariant distributions}

\begin{theorem}\label{thm_infinv}
Let \(\L\) be a Lévy process in \(\real\) with characteristic triplet \((\gamma,\sigma^2,\Pi)\) with \(\Pi=\mu+\rho\) as in Section \ref{sec_prel}. Let \(\ppi\) be a distribution on \(\real\) such that \textbf{(a1)} - \textbf{(a3)} are fulfilled. Consider the \textsc{sde} \eqref{eq_sde} with
\begin{align}\label{eq_drift_coeff}
\phi(x):=\bone_{\cE}(x)\left(\frac{\frac{1}2\sigma^2 \pi'(x)- \overline\mu_s *\pi(x)+ (\overline{\overline\rho}*\pi)'(x)}{\pi(x)}-\gamma\right).
\end{align}
Then \(\ppi\) is an infinitesimally invariant distribution of any solution \(\X\) of \eqref{eq_sde}.  
\end{theorem}

The proof of Theorem \ref{thm_infinv} will be clearer if we point out the primary thoughts behind Assumptions \textbf{(a1)} - \textbf{(a3)} first. 
\begin{remark}
Assumptions \textbf{(a1)} and \textbf{(a2)} make sure that the process \(\X\) stays in the open half-line \((0,\infty)\) if \(\cE=(0,\infty)\). The former does so by allowing only upward jumps while the latter guarantees that \(\X\) cannot drift onto \(0\), as we will see in the proof below.\\ 
Clearly, Assumption \textbf{(a3)} becomes only relevant if \(\cE=\real\), and it ensures that our choice of the drift coefficient in \eqref{eq_drift_coeff} is well-defined. Note that it can be weakened if \(\sigma^2=0\) as \(\pi\in\sobolev\) is sufficient but not necessary for \((\overline{\overline\rho}*\pi)'\) to be well-defined. However, since we discuss in this article mostly processes with paths of bounded variation we choose to omit various special cases for the sake of clarity.
\end{remark}

\begin{proof}[Proof of Theorem \ref{thm_infinv}]
Denote by \(\mathcal O\subset\real\) the state space of \(\X\). In order to show that \(\cO=\cE\) we prove that \(X_t\in\cE\) for all \(t\geq0\) if \(X_0\in\cE\). As this is trivially true for \(\cE=\real\) we show it only for \(\cE=(0,\infty)\). \\
In this case \(\L\) is a spectrally positive compound Poisson process, by \textbf{(a1)}. Thus, \(\X\) cannot exit \(\cE\) via jumps. We are going to show that \(\X\) cannot exit via drift either. If no jump of \(\L\) interrupts the path of \(\X\) then \(t\mapsto X_t\) is monotone decreasing and follows the autonomous differential equation
\begin{align*}
\begin{cases}
\diff X_t=\phi(X_t)\diff t,\\
X_0=x>0,
\end{cases}
\end{align*}
with \(\phi(x)=-\bone_{(0,\infty)}(x)\frac{\overline\mu_s*\pi(x)}{\pi(x)}\). Separation of variables yields that the time \(T\) it takes for \(\X\) to drift from \(x\) to \(x'\in[0,x]\) is given by
\begin{align*}
T= \int_{x'}^x \frac{\pi(z)}{\overline\mu_s*\pi(z)}\diff z.
\end{align*}
By \textbf{(a1)} and \textbf{(a2)}, \(\overline\mu_s*\pi(x)\leq cx\pi(x)\) for some constant \(c>0\), and thereby
\begin{align}\label{eq_subor}
\int_0^x \frac{\pi(z)}{\overline\mu_s*\pi(z)} \diff z =\infty
\end{align}
for all \(x>0\). Hence, \(\X\) cannot drift onto \(0\) in finite time. Therefore, \(\cO=\cE\) in this case as well.\\
We now return to the general case. A straight-forward application of Itô's lemma and the Lévy-Itô decomposition, similar to \cite[Thm. 2.50]{schnurr2009symbol}, yields that for the pointwise generator   \((\cA,\domain)\) of \(\X\) it holds \(\cC_c^\infty(\mathcal O)\subset\domain\), and \begin{align*}
\cA f(x)&= (\phi(x)+\gamma)f'(x) + \frac12 \sigma^2 f''(x)\\
&\qquad+ \int_{\real} (f(x+z)-f(x))\mu(\diff z)\\
&\qquad + \int_{\real} (f(x+z)-f(x)- \nabla f(x))\rho(\diff z)
\end{align*}
for all \(f\in\cC_c^\infty(\mathcal O)\). By \cite[Thm. 4.2]{behmeoechsler} a measure \(\eeta\) is infinitesimally invariant for \(\X\) if
\begin{align}\label{eq_iile}
-((\phi+\gamma) \eeta)' +   \frac12\sigma^2  \eeta''- (\overline\mu_s*\eeta)' +  (\overline{\overline\rho}*\eeta)'' = 0
\end{align}
in the distributional sense w.r.t. \(\cC_c^\infty(\mathcal O)\). Because of \(\cO=\cE\), simply inserting \(\ppi\) into \eqref{eq_iile} proves the claim.
\end{proof}
 
\subsection{Invariant distributions}

In general, proving that an infinitesimally invariant distribution is an invariant distribution is hard. The best-case scenario is given when \(\X\) is a Feller process and the test functions constitute a core of the pointwise generator of \(\X\). In this case infinitesimally invariant and invariant are equivalent notions, cf. \cite{liggett}.\\
Although there exist easily verifiable conditions on the drift coefficient \(\phi\) and \(\L\) such that a solution of \eqref{eq_sde} is a Feller process (cf. \cite{kuhn}) these have some drawbacks. The fact that typically, \(\phi\) is required to be continuous and fulfills a linear growth condition, i.e. \(|\phi(x)|\leq C(1+|x|)\) for some \(C>0\), excludes many interesting cases. Moreover, even if \(\X\) is a Feller process, we are still left with the question whether the test functions form a core. The task of finding conditions for this to be the case is an open problem (cf. \cite{bottcher}) which has not yet been answered to the best of our knowledge.

\begin{remark}\label{rem_source}
Both the article \cite{eliazar} on Lévy Langevin dynamics and the original article \cite{FLI} on f\textsc{lmc} do not provide arguments as to why the considered target measures are invariant for the respective processes.\\
The chosen appraoch in both articles revolves around finding a stationary solution for Kolmogorov's forward equation of the underlying \textsc{sde} \eqref{eq_sde}. This equation, which is also known as Fokker-Planck equation, is inherently connected to invariant distributions as any weak stationary solution of it can be associated to an infinitesimally invariant measure of a solution \(\X\) of \eqref{eq_sde}. In the aforementioned articles it is suggested that the transition densities \(p(t,x,y)\) of \(\X\) defined via
\begin{align*}
\PP^x(X_t\in B)=\int_B p(t,x,y)\diff y
\end{align*}
solve the associated Kolmogorov forward equation. For many processes this is true, e.g. Feller diffusions (cf. \cite{kallenberg}) just to name one. However, for \textsc{sde}s \eqref{eq_sde} with general Lévy noises we were not able to find a reference with a rigorous proof of this claim. If it was indeed true, then any invariant measure \(\ppi\) of \(\X\) would necessarily be a stationary solution of Kolmogorov's forward equation. \\
Moreover, both articles are missing an argument as to why the stationary solution of Kolmogorov's forward equation is unique. Although in \cite{FLI} another article (cf. \cite{schertzer2001fractional}) is cited on this topic, in said reference uniqueness is merely argued heuristically but not proved. 
\end{remark}

To ensure methodological rigor, we present in this section a different way of showing that \(\ppi\) is an invariant distribution of \(\X\), and as such even unique. We consider this approach in the special case of \(\cE=(0,\infty)\) and \(\L\) being a compound Poisson process but similar results can be achieved in other frameworks by adjusting the individual steps in a suitable way. This will also be discussed in Section \ref{sec_exam} below.\\

\noindent
Let us now briefly describe our setting. Denote by
\begin{align*}
\mathcal P:=\{(x_i)_{i\in\integer}\subset (0,\infty)^\integer:~& x_i<x_{i+1} \text{ for all }i\in\integer, \text{ and}\\
&0 \text{ is the unique accumulation point of } (x_i)_{i\in\integer}\}
\end{align*} 
a set of partitions of the open interval \((0,\infty)\). We call a function \(f\in L^1_\loc(0,\infty)\) \emph{piecewise weakly differentiable} if there exists a partition \((x_i)_{i\in\integer}\in \mathcal P\) such that \(f|_{(x_i,x_i+1)}\in W^{1,1}(x_i,x_{i+1})\) for all \(i\in\integer\). Analogously, we call \(f\) \emph{piecewise Lipschitz continuous} if there exists a partition \((x_i)_{i\in\integer}\in\mathcal P\) such that \(f|_{(x_i,x_i+1)}\) is Lipschitz continuous for all \(i\in\integer\). \\
Let \(\L\) be a Lévy process in \(\real\) and let \(\ppi(\diff x)=\pi(x)\diff x\) be an absolutely continuous distribution on \((0,\infty)\). Our assumptions are as follows:
\begin{enumerate}
\item[\textbf{(b1)}] \(\pi\) is a positive, piecewise weakly differentiable function, and there exist constants \(C,C',\alpha>0\) such that \(\lim_{x\to\infty}\pi(x)\ee^{\alpha x}=C\), and \(\int_0^x \pi(z)\diff z\leq C' \pi(x)x\) for \(x\ll1\).
\item[\textbf{(b2)}] \(\L\) is a spectrally positive compound Poisson process, i.e. a Lévy process with characteristic triplet \((0,0,\mu)\) such that \(\supp\mu\subset\real_+\) and \(\int_0^\infty (1 \vee z)\mu(\diff z)<\infty\).
\end{enumerate}
Note that our standing assumptions \textbf{(a1)} - \textbf{(a3)} are direct consequences of \textbf{(b1)} and \textbf{(b2)}. In this setting, the drift coefficient given by \eqref{eq_drift_coeff} is reduced to
\begin{align}\label{eq_driftcpn}
\phi(x)= -\bone_{(0,\infty)}(x)\frac{\overline\mu_s*\pi(x)}{\pi(x)},
\end{align}
and it is easy to see that \(\phi(x)\in(-\infty,0)\) for all \(x>0\). \\
Sometimes it will be advantageous to write \(L_t=\sum_{i=1}^{N_t}\xi_i\), where \(\N\) is a Poisson process with intensity \(|\mu|\) and \((\xi_i)_{i\in\nat}\) is a sequence of i.i.d. random variables distributed according to \(\mu/|\mu|\). Note that \(\EE\xi_1<\infty\) by \textbf{(b2)}.\\
With the following theorem we show that, under \textbf{(b1)} and \textbf{(b2)}, a solution \(\X\) of \eqref{eq_sde} has the unique invariant distribution \(\ppi\) if, additionally, one of the following two conditions is met:
\begin{enumerate}
\item[\textbf{(c1)}] There exists \(n\in\nat\) such that \(\supp \mu \subset(1/n,n)\), or
\item[\textbf{(c2)}] \(\pi\) is piecewise Lipschitz continuous.
\end{enumerate} 

\begin{theorem}\label{thm_cpn}
Assume that \textbf{(b1)} and \textbf{(b2)} hold, and let \(\X\) be a solution of \eqref{eq_sde} with \(\phi\) as in \eqref{eq_driftcpn}. Then
\begin{enumerate}
\item \(\X\) is positive Harris recurrent, and
\item any invariant distribution of \(\X\) is an infinitesimally invariant distribution of \(\X\).
\end{enumerate}
Additionally, if \textbf{(c1)} or \textbf{(c2)} are fulfilled, then
\begin{enumerate}
\item [(iii)] \(\pi\) is the unique invariant distribution of \(\X\).
\end{enumerate}
\end{theorem}

The proof of Theorem \ref{thm_cpn} is presented in Section \ref{sec_prf}, and it is divided into several steps. The first assertion is shown by using the Foster-Lyapunov method of \cite{MTI} - \cite{MTIII}, while the second assertion is a simple application of \cite[Cor. 5.4]{behmeoechsler}. Under \textbf{(c1)} we show Theorem \ref{thm_cpn} (iii) via techniques from the theory of ordinary differential equations. If instead \textbf{(c2)} is true, we approximate \(\X\) by a sequence of processes fulfilling \textbf{(c1)} to prove the claim.

\subsection{Limiting distributions}

The natural follow-up question of Theorem \ref{thm_cpn} is whether existence and uniqueness of an invariant distribution \(\eta\) for \(\X\) implies that \(\eta\) is a limiting distribution. This property of \(\X\), i.e. the existence of a limiting distribution, is called \emph{ergodicity}.\\
As before, \(\ppi(\diff x)=\pi(x)\diff x\), \(\cE=(0,\infty)\), and \(\L\) is a spectrally positive compound Poisson process.

\begin{corollary}\label{cor_erg}
Assume \textbf{(b1)} and \textbf{(b2)} hold, and let \(\X\) be a solution of \eqref{eq_sde} with \(\phi\) as in \eqref{eq_driftcpn}. Further assume that some skeleton chain of \(\X\) is irreducible, i.e. there exits \(\Delta>0\) such that for all \(B\in\cB((0,\infty))\) with \(\lambda^{\mathrm{Leb}}(B)>0\) and all \(x\in(0,\infty)\) there exists \(n\in\nat\) such that
\begin{align*}
\PP^x(X_{n\Delta}\in B)>0.
\end{align*} 
Then \(\X\) is \emph{ergodic}.
\end{corollary}
\begin{proof}
Follows directly from Theorem \ref{thm_cpn} (i) and \cite[Thm. 6.1]{MTII}.
\end{proof}

\begin{lemma}[Irreducible skeleton chain]\label{lem_full}
Assume \textbf{(b1)} and \textbf{(b2)} hold, and let \(\X\) be a solution of \eqref{eq_sde} with \(\phi\) as in \eqref{eq_driftcpn}. Additionally assume that \(\mu=\mu_1+\mu_2\) where \(\mu_1\) is arbitrary and \(\mu_2\) is absolutely continuous and such that \(\mu_2(I)>0\) for all open intervals \(I\subset(0,\infty)\). Then the 1-skeleton chain is irreducible, and \(\X\) is ergodic.
\end{lemma}
\begin{proof}
Without loss of generality we assume \(\mu_1=0\) since otherwise we may simply condition on the event that the jumps are only sampled from \(\mu_2\). \\
Let \(B\in\cB((0,\infty))\) with \(\lambda^{\mathrm{Leb}}(B)>0\). Our goal is to show that for all \(x\in(0,\infty)\) there exists \(n\in\nat\) such that  
\begin{align}\label{eq_ir}
\PP^x(X_{n}\in B)>0.
\end{align} 
It suffices to show \eqref{eq_ir} only for sets \(B\) for which \(\inf B>0\) since for arbitrary \(B\in\cB(0,\infty)\) with \(\lambda^{\mathrm{Leb}}(B)>0\) there exist \(0<a<b\) such that \(\lambda^{\mathrm{Leb}}(B\cap(a,b))>0\).
Moreover, it also suffices to just consider \(x<\inf B\) and \(n=1\). This is due to the fact that for arbitrary \(x\in(0,\infty)\) and \(m\in\nat\) we obtain
\begin{align*}
\PP^x(X_{m}\in B)\geq \PP^x(X_m\in B, N_{m-1}=0),
\end{align*}
where we recall that \(\N\) is the Poisson process counting the jumps of \(\L\), and therefore also of \(\X\).\\
As \(\phi(x)<0\) for all \(x>0\), we may choose \(m\) large enough such that \(X^x_{m-1}<\inf B\) on \(\{N_{m-1}=0\}\), and consider \(X^x_{m-1}\) as a new starting point.\\
Thus, let \(0<x<\inf B\). In the following we condition on the event that exactly one jump occurs until \(t=1\). Denote \(Y_t:=\left(X^x_t\big|N_1=1\right)\). It holds
\begin{align*}
\PP^x(X_1\in B)\geq c\PP^x(Y_1\in B)
\end{align*}
for some \(c>0\). Further, denote by \(T\in(0,1)\) the uniformly distributed time of the jump. We show that the joint cumulative distribution function
\begin{align*}
\PP(T\leq t, Y_1\leq y)
\end{align*}
is strictly monotone on \((0,1)\times (x,\infty)\) in both arguments. Let \(0<t<t'<1\) and \(y\in(x,\infty)\). We obtain
\begin{align*}
\PP(T\in(t,t'], Y_1\leq y)\geq \PP(T\in(t,t'], \xi_1\leq x-Y_{t-}) >0.
\end{align*}
Indeed, if \(T\in(t,t']\) and additionally \(\xi_1\leq x-Y_{t-}\), then \(Y_1\leq x<y\), and since \(Y_{t-}<x\) we get \(\PP(T\in(t,t'],\xi_1\leq x-Y_{t-})>0\).\\
 Now, let \(t\in(0,1)\) and \(x<y<y'<\infty\). We note that for every \(t\in(0,1)\) there exists some interval \(I\subset (0,\infty)\) such that \(Y_1\in(y,y']\) if \(T=t\) and \(\xi_1\in I\). This is due to the fact that the paths of \(\X\) between two jumps are continuous and strictly decreasing. Moreover, since \(Y_1\) depends continuously on \(T\) and \(\xi_1\), there exists \(\varepsilon>0\) and an interval \(I'\subset(0,\infty)\) such that \(Y_1\in(y,y']\) if \(T\in(t-\varepsilon,t]\) and \(\xi_1\in I'\). Thus,
\begin{align*}
\PP(T\leq t, Y_1\in(y,y']) \geq\PP(T\in(t-\varepsilon,t],Y_1\in(y,y'])\geq \varepsilon\mu(I')>0
\end{align*}
by the assumption on \(\mu\).\\
As both \(T\) and \(Y_1\) have clearly no atoms in \((0,1)\) and \((x,\infty)\), respectively, there exists a joint density function \(f_{(T,Y_1)}\) of \((T,Y_1)\) on \((0,1)\times(x,\infty)\) which is strictly positive. Hence
\begin{align*}
\PP^x(X_1\in B)\geq c\PP^x(T\in(0,1),Y_1\in B)=\int_{(0,1)\times B} f_{(T,Y_1)}(t,y)\diff t\diff y >0.
\end{align*}
This, together with Corollary \ref{cor_erg}, concludes the proof.
\end{proof}

\subsection{Examples}\label{sec_exam}
In this section we illustrate Theorem \ref{thm_cpn} on various examples by sampling \eqref{eq_sde}. To this end, we first compute a realization of the path of the driving noise \(\L\). With \(\L\) being a compound Poisson process this is straight-forward. It remains to solve a (deterministic) differential equation which is then done via the classic Euler method.

\noindent
\begin{figure}
\begin{tikzpicture}
\begin{axis}[
	  height=8em,
 	  width=21em,
 	  area style,
      tick label style={font=\tiny},
 	  ytick={0,1},
	  xmin=0,
	  xmax=10,
	  ymin=0,
	  ymax=1]
		\addplot[line width=2, color=black] coordinates{( 0.01 , 0.0 )
( 0.06 , 0.0 )
( 0.11 , 0.0 )
( 0.16 , 1e-05 )
( 0.21 , 3e-05 )
( 0.26 , 8e-05 )
( 0.31 , 0.0002 )
( 0.36 , 0.00046 )
( 0.41 , 0.00099 )
( 0.46 , 0.00204 )
( 0.51 , 0.00398 )
( 0.56 , 0.00742 )
( 0.61 , 0.0132 )
( 0.66 , 0.02245 )
( 0.71 , 0.03654 )
( 0.76 , 0.05701 )
( 0.81 , 0.08538 )
( 0.86 , 0.12287 )
( 0.91 , 0.17015 )
( 0.96 , 0.22699 )
( 1.01 , 0.29208 )
( 1.06 , 0.36296 )
( 1.11 , 0.4361 )
( 1.16 , 0.50723 )
( 1.21 , 0.57175 )
( 1.26 , 0.62531 )
( 1.31 , 0.66429 )
( 1.36 , 0.68625 )
( 1.41 , 0.69016 )
( 1.46 , 0.67643 )
( 1.51 , 0.64681 )
( 1.56 , 0.60404 )
( 1.61 , 0.55149 )
( 1.66 , 0.49277 )
( 1.71 , 0.43134 )
( 1.76 , 0.37026 )
( 1.81 , 0.31197 )
( 1.86 , 0.25827 )
( 1.91 , 0.21028 )
( 1.96 , 0.16854 )
( 2.01 , 0.1331 )
( 2.06 , 0.10366 )
( 2.11 , 0.07968 )
( 2.16 , 0.06052 )
( 2.21 , 0.04544 )
( 2.26 , 0.03377 )
( 2.31 , 0.02486 )
( 2.36 , 0.01814 )
( 2.41 , 0.01313 )
( 2.46 , 0.00943 )
( 2.51 , 0.00674 )
( 2.56 , 0.00478 )
( 2.61 , 0.00338 )
( 2.66 , 0.00238 )
( 2.71 , 0.00167 )
( 2.76 , 0.00117 )
( 2.81 , 0.00081 )
( 2.86 , 0.00057 )
( 2.91 , 0.00039 )
( 2.96 , 0.00027 )
( 3.01 , 0.00019 )
( 3.06 , 0.00013 )
( 3.11 , 9e-05 )
( 3.16 , 7e-05 )
( 3.21 , 5e-05 )
( 3.26 , 3e-05 )
( 3.31 , 2e-05 )
( 3.36 , 2e-05 )
( 3.41 , 1e-05 )
( 3.46 , 1e-05 )
( 3.51 , 1e-05 )
( 3.56 , 0.0 )
( 3.61 , 0.0 )
( 3.66 , 0.0 )
( 3.71 , 0.0 )
( 3.76 , 0.0 )
( 3.81 , 0.0 )
( 3.86 , 0.0 )
( 3.91 , 0.0 )
( 3.96 , 0.0 )
( 4.01 , 0.0 )
( 4.06 , 0.0 )
( 4.11 , 0.0 )
( 4.16 , 0.0 )
( 4.21 , 0.0 )
( 4.26 , 0.0 )
( 4.31 , 0.0 )
( 4.36 , 0.0 )
( 4.41 , 0.0 )
( 4.46 , 0.0 )
( 4.51 , 0.0 )
( 4.56 , 0.0 )
( 4.61 , 0.0 )
( 4.66 , 0.0 )
( 4.71 , 0.0 )
( 4.76 , 0.0 )
( 4.81 , 0.0 )
( 4.86 , 0.0 )
( 4.91 , 0.0 )
( 4.96 , 0.0 )
( 5.01 , 0.0 )
( 5.06 , 0.0 )
( 5.11 , 0.0 )
( 5.16 , 0.0 )
( 5.21 , 0.0 )
( 5.26 , 0.0 )
( 5.31 , 0.0 )
( 5.36 , 0.0 )
( 5.41 , 0.0 )
( 5.46 , 0.0 )
( 5.51 , 0.0 )
( 5.56 , 0.0 )
( 5.61 , 0.0 )
( 5.66 , 0.0 )
( 5.71 , 0.0 )
( 5.76 , 0.0 )
( 5.81 , 0.0 )
( 5.86 , 0.0 )
( 5.91 , 0.0 )
( 5.96 , 0.0 )
( 6.01 , 0.0 )
( 6.06 , 0.0 )
( 6.11 , 0.0 )
( 6.16 , 0.0 )
( 6.21 , 0.0 )
( 6.26 , 0.0 )
( 6.31 , 0.0 )
( 6.36 , 0.0 )
( 6.41 , 0.0 )
( 6.46 , 0.0 )
( 6.51 , 1e-05 )
( 6.56 , 1e-05 )
( 6.61 , 1e-05 )
( 6.66 , 2e-05 )
( 6.71 , 2e-05 )
( 6.76 , 3e-05 )
( 6.81 , 4e-05 )
( 6.86 , 6e-05 )
( 6.91 , 9e-05 )
( 6.96 , 0.00013 )
( 7.01 , 0.00019 )
( 7.06 , 0.00027 )
( 7.11 , 0.00038 )
( 7.16 , 0.00055 )
( 7.21 , 0.00079 )
( 7.26 , 0.00112 )
( 7.31 , 0.0016 )
( 7.36 , 0.00228 )
( 7.41 , 0.00323 )
( 7.46 , 0.00456 )
( 7.51 , 0.00639 )
( 7.56 , 0.00892 )
( 7.61 , 0.01237 )
( 7.66 , 0.01702 )
( 7.71 , 0.02323 )
( 7.76 , 0.03143 )
( 7.81 , 0.0421 )
( 7.86 , 0.05581 )
( 7.91 , 0.07314 )
( 7.96 , 0.09467 )
( 8.01 , 0.12094 )
( 8.06 , 0.15234 )
( 8.11 , 0.18904 )
( 8.16 , 0.23087 )
( 8.21 , 0.27725 )
( 8.26 , 0.32706 )
( 8.31 , 0.37865 )
( 8.36 , 0.4298 )
( 8.41 , 0.47784 )
( 8.46 , 0.51981 )
( 8.51 , 0.55271 )
( 8.56 , 0.57385 )
( 8.61 , 0.58113 )
( 8.66 , 0.57342 )
( 8.71 , 0.55071 )
( 8.76 , 0.5142 )
( 8.81 , 0.46625 )
( 8.86 , 0.4101 )
( 8.91 , 0.34951 )
( 8.96 , 0.28827 )
( 9.01 , 0.22983 )
( 9.06 , 0.17692 )
( 9.11 , 0.13133 )
( 9.16 , 0.0939 )
( 9.21 , 0.06458 )
( 9.26 , 0.04267 )
( 9.31 , 0.02706 )
( 9.36 , 0.01644 )
( 9.41 , 0.00956 )
( 9.46 , 0.00531 )
( 9.51 , 0.00282 )
( 9.56 , 0.00143 )
( 9.61 , 0.00069 )
( 9.66 , 0.00031 )
( 9.71 , 0.00014 )
( 9.76 , 6e-05 )
( 9.81 , 2e-05 )
( 9.86 , 1e-05 )
( 9.91 , 0.0 )
( 9.96 , 0.0 )};
        \end{axis}
        \end{tikzpicture}\begin{tikzpicture}
\begin{axis}[
	  height=8em,
 	  width=21em,
 	  area style,
      tick label style={font=\tiny},
 	  ytick={0,0.5,1},
	  xmin=0,
	  xmax=10,
	  ymin=0,
	  ymax=0.5]
		\addplot[line width=2, color=black] coordinates{( 0.01 , 0.249 )
( 0.06 , 0.243 )
( 0.11 , 0.237 )
( 0.16 , 0.231 )
( 0.21 , 0.225 )
( 0.26 , 0.22 )
( 0.31 , 0.214 )
( 0.36 , 0.209 )
( 0.41 , 0.204 )
( 0.46 , 0.199 )
( 0.51 , 0.194 )
( 0.56 , 0.189 )
( 0.61 , 0.184 )
( 0.66 , 0.18 )
( 0.71 , 0.175 )
( 0.76 , 0.171 )
( 0.81 , 0.167 )
( 0.86 , 0.163 )
( 0.91 , 0.159 )
( 0.96 , 0.155 )
( 1.01 , 0.151 )
( 1.06 , 0.147 )
( 1.11 , 0.144 )
( 1.16 , 0.14 )
( 1.21 , 0.137 )
( 1.26 , 0.133 )
( 1.31 , 0.13 )
( 1.36 , 0.127 )
( 1.41 , 0.124 )
( 1.46 , 0.121 )
( 1.51 , 0.118 )
( 1.56 , 0.115 )
( 1.61 , 0.112 )
( 1.66 , 0.109 )
( 1.71 , 0.106 )
( 1.76 , 0.104 )
( 1.81 , 0.101 )
( 1.86 , 0.099 )
( 1.91 , 0.096 )
( 1.96 , 0.094 )
};
		\addplot[line width=2, color=black] coordinates{( 2.01 , 0.342 )
( 2.06 , 0.339 )
( 2.11 , 0.337 )
( 2.16 , 0.335 )
( 2.21 , 0.333 )
( 2.26 , 0.331 )
( 2.31 , 0.329 )
( 2.36 , 0.327 )
( 2.41 , 0.325 )
( 2.46 , 0.323 )
( 2.51 , 0.321 )
( 2.56 , 0.32 )
( 2.61 , 0.318 )
( 2.66 , 0.316 )
( 2.71 , 0.315 )
( 2.76 , 0.313 )
( 2.81 , 0.312 )
( 2.86 , 0.31 )
( 2.91 , 0.309 )
( 2.96 , 0.307 )
( 3.01 , 0.306 )
( 3.06 , 0.304 )
( 3.11 , 0.303 )
( 3.16 , 0.302 )
( 3.21 , 0.3 )
( 3.26 , 0.299 )
( 3.31 , 0.298 )
( 3.36 , 0.297 )
( 3.41 , 0.296 )
( 3.46 , 0.295 )
( 3.51 , 0.293 )
( 3.56 , 0.292 )
( 3.61 , 0.291 )
( 3.66 , 0.29 )
( 3.71 , 0.289 )
( 3.76 , 0.288 )
( 3.81 , 0.287 )
( 3.86 , 0.286 )
( 3.91 , 0.286 )
( 3.96 , 0.285 )};
		\addplot[line width=2, color=black] coordinates{
( 4.01 , 0.034 )
( 4.06 , 0.033 )
( 4.11 , 0.032 )
( 4.16 , 0.031 )
( 4.21 , 0.03 )
( 4.26 , 0.03 )
( 4.31 , 0.029 )
( 4.36 , 0.028 )
( 4.41 , 0.028 )
( 4.46 , 0.027 )
( 4.51 , 0.026 )
( 4.56 , 0.026 )
( 4.61 , 0.025 )
( 4.66 , 0.024 )
( 4.71 , 0.024 )
( 4.76 , 0.023 )
( 4.81 , 0.023 )
( 4.86 , 0.022 )
( 4.91 , 0.021 )
( 4.96 , 0.021 )
( 5.01 , 0.02 )
( 5.06 , 0.02 )
( 5.11 , 0.019 )
( 5.16 , 0.019 )
( 5.21 , 0.018 )
( 5.26 , 0.018 )
( 5.31 , 0.018 )
( 5.36 , 0.017 )
( 5.41 , 0.017 )
( 5.46 , 0.016 )
( 5.51 , 0.016 )
( 5.56 , 0.016 )
( 5.61 , 0.015 )
( 5.66 , 0.015 )
( 5.71 , 0.014 )
( 5.76 , 0.014 )
( 5.81 , 0.014 )
( 5.86 , 0.013 )
( 5.91 , 0.013 )
( 5.96 , 0.013 )
( 6.01 , 0.012 )
( 6.06 , 0.012 )
( 6.11 , 0.012 )
( 6.16 , 0.011 )
( 6.21 , 0.011 )
( 6.26 , 0.011 )
( 6.31 , 0.011 )
( 6.36 , 0.01 )
( 6.41 , 0.01 )
( 6.46 , 0.01 )
( 6.51 , 0.01 )
( 6.56 , 0.009 )
( 6.61 , 0.009 )
( 6.66 , 0.009 )
( 6.71 , 0.009 )
( 6.76 , 0.009 )
( 6.81 , 0.008 )
( 6.86 , 0.008 )
( 6.91 , 0.008 )
( 6.96 , 0.008 )
( 7.01 , 0.008 )
( 7.06 , 0.007 )
( 7.11 , 0.007 )
( 7.16 , 0.007 )
( 7.21 , 0.007 )
( 7.26 , 0.007 )
( 7.31 , 0.006 )
( 7.36 , 0.006 )
( 7.41 , 0.006 )
( 7.46 , 0.006 )
( 7.51 , 0.006 )
( 7.56 , 0.006 )
( 7.61 , 0.006 )
( 7.66 , 0.005 )
( 7.71 , 0.005 )
( 7.76 , 0.005 )
( 7.81 , 0.005 )
( 7.86 , 0.005 )
( 7.91 , 0.005 )
( 7.96 , 0.005 )
( 8.01 , 0.005 )
( 8.06 , 0.004 )
( 8.11 , 0.004 )
( 8.16 , 0.004 )
( 8.21 , 0.004 )
( 8.26 , 0.004 )
( 8.31 , 0.004 )
( 8.36 , 0.004 )
( 8.41 , 0.004 )
( 8.46 , 0.004 )
( 8.51 , 0.004 )
( 8.56 , 0.003 )
( 8.61 , 0.003 )
( 8.66 , 0.003 )
( 8.71 , 0.003 )
( 8.76 , 0.003 )
( 8.81 , 0.003 )
( 8.86 , 0.003 )
( 8.91 , 0.003 )
( 8.96 , 0.003 )
( 9.01 , 0.003 )
( 9.06 , 0.003 )
( 9.11 , 0.003 )
( 9.16 , 0.003 )
( 9.21 , 0.003 )
( 9.26 , 0.002 )
( 9.31 , 0.002 )
( 9.36 , 0.002 )
( 9.41 , 0.002 )
( 9.46 , 0.002 )
( 9.51 , 0.002 )
( 9.56 , 0.002 )
( 9.61 , 0.002 )
( 9.66 , 0.002 )
( 9.71 , 0.002 )
( 9.76 , 0.002 )
( 9.81 , 0.002 )
( 9.86 , 0.002 )
( 9.91 , 0.002 )
( 9.96 , 0.002 )};
        \end{axis}
        \end{tikzpicture}\\
        \begin{tikzpicture}
      \begin{axis}[
	  height=8em,
 	  width=21em,
 	  area style,
      tick label style={font=\tiny},
 	  ytick={0,1},
	  xmin=0,
	  xmax=10,
	  ymin=0,
	  ymax=1]
	  \addplot[ ybar interval,
	  fill,
                color=black!70
                ] coordinates {( 0.569 , 0.00307 )
( 0.646 , 0.01554 )
( 0.722 , 0.03742 )
( 0.798 , 0.07025 )
( 0.875 , 0.12626 )
( 0.951 , 0.21291 )
( 1.027 , 0.32749 )
( 1.104 , 0.44942 )
( 1.18 , 0.55361 )
( 1.256 , 0.64205 )
( 1.333 , 0.70278 )
( 1.409 , 0.71815 )
( 1.485 , 0.67879 )
( 1.561 , 0.61649 )
( 1.638 , 0.52332 )
( 1.714 , 0.43268 )
( 1.79 , 0.34455 )
( 1.867 , 0.26181 )
( 1.943 , 0.17174 )
( 2.019 , 0.12439 )
( 2.096 , 0.0744 )
( 2.172 , 0.07395 )
( 2.248 , 0.02504 )
( 2.325 , 0.00021 )
( 2.401 , 0.00021 )
( 2.477 , 0.00465 )
( 2.554 , 0.00507 )
( 2.63 , 0.00021 )
( 2.706 , 0.00024 )
( 2.783 , 0.0001 )
( 2.859 , 0.00068 )
( 2.935 , 8e-05 )
( 3.012 , 8e-05 )
( 3.088 , 0.00018 )
( 3.164 , 8e-05 )
( 3.241 , 0.0001 )
( 3.317 , 5e-05 )
( 3.393 , 3e-05 )
( 3.469 , 8e-05 )
( 3.546 , 8e-05 )
( 3.622 , 0.00021 )
( 3.698 , 3e-05 )
( 3.775 , 5e-05 )
( 3.851 , 3e-05 )
( 3.927 , 0.0001 )
( 4.004 , 5e-05 )
( 4.08 , 5e-05 )
( 4.156 , 3e-05 )
( 4.233 , 0.0 )
( 4.309 , 5e-05 )
( 4.385 , 3e-05 )
( 4.462 , 3e-05 )
( 4.538 , 5e-05 )
( 4.614 , 0.0 )
( 4.691 , 5e-05 )
( 4.767 , 0.0001 )
( 4.843 , 0.00021 )
( 4.92 , 0.00031 )
( 4.996 , 0.00029 )
( 5.072 , 0.00034 )
( 5.149 , 0.0006 )
( 5.225 , 0.00063 )
( 5.301 , 0.00089 )
( 5.377 , 0.00068 )
( 5.454 , 0.00092 )
( 5.53 , 0.00076 )
( 5.606 , 0.00047 )
( 5.683 , 0.00039 )
( 5.759 , 0.00037 )
( 5.835 , 0.00021 )
( 5.912 , 0.00026 )
( 5.988 , 0.00021 )
( 6.064 , 8e-05 )
( 6.141 , 5e-05 )
( 6.217 , 0.0 )
( 6.293 , 8e-05 )
( 6.37 , 8e-05 )
( 6.446 , 0.0001 )
( 6.522 , 5e-05 )
( 6.599 , 0.00073 )
( 6.675 , 0.00076 )
( 6.751 , 0.00102 )
( 6.828 , 0.0011 )
( 6.904 , 0.0015 )
( 6.98 , 0.00234 )
( 7.057 , 0.0026 )
( 7.133 , 0.00144 )
( 7.209 , 0.00249 )
( 7.286 , 0.00378 )
( 7.362 , 0.00496 )
( 7.438 , 0.00588 )
( 7.514 , 0.01018 )
( 7.591 , 0.01404 )
( 7.667 , 0.01921 )
( 7.743 , 0.03265 )
( 7.82 , 0.05094 )
( 7.896 , 0.07002 )
( 7.972 , 0.1072 )
( 8.049 , 0.15355 )
( 8.125 , 0.20748 )
( 8.201 , 0.28151 )
( 8.278 , 0.3592 )
( 8.354 , 0.44139 )
( 8.43 , 0.51338 )
( 8.507 , 0.57122 )
( 8.583 , 0.601 )
( 8.659 , 0.5857 )
( 8.736 , 0.52574 )
( 8.812 , 0.44588 )
( 8.888 , 0.34623 )
( 8.965 , 0.24278 )
( 9.041 , 0.15352 )
( 9.117 , 0.08679 )
( 9.194 , 0.0333 )
( 9.27 , 0.00194 )
( 9.346 , 0.00171 )
( 9.422 , 0.00192 )
( 9.499 , 0.00134 )
( 9.575 , 0.00139 )
( 9.651 , 0.00108 )
( 9.728 , 0.00087 )
( 9.804 , 0.00068 )
( 9.88 , 0.00042 )
( 9.957 , 0.00058 )
( 10.033 , 0.00029 )
( 10.109 , 0.00031 )
( 10.186 , 0.00013 )
( 10.262 , 0.00018 )
( 10.338 , 8e-05 )
( 10.415 , 5e-05 )
( 10.491 , 0.00016 )
( 10.567 , 5e-05 )
( 10.644 , 0.0001 )
( 10.72 , 8e-05 )
( 10.796 , 0.0001 )
( 10.873 , 3e-05 )
( 10.949 , 3e-05 )
( 11.025 , 0.0 )
( 11.102 , 3e-05 )
( 11.178 , 5e-05 )
( 11.254 , 3e-05 )
( 11.33 , 3e-05 )
( 11.407 , 3e-05 )
( 11.483 , 3e-05 )
( 11.559 , 3e-05 )
( 11.636 , 5e-05 )
( 11.712 , 5e-05 )
( 11.788 , 5e-05 )
( 11.865 , 3e-05 )
( 11.941 , 0.0001 )};
        \end{axis}
    \end{tikzpicture}\begin{tikzpicture}
      \begin{axis}[
	  height=8em,
 	  width=21em,
 	  area style,
      tick label style={font=\tiny},
 	  ytick={0,0.5},
	  xmin=0,
	  xmax=10,
	  ymin=0,
	  ymax=0.5]
	  \addplot[ ybar interval,
	  fill,
                color=black!70
                ] coordinates {( 0.05 , 0.284 )
( 0.15 , 0.234 )
( 0.25 , 0.228 )
( 0.35 , 0.221 )
( 0.45 , 0.192 )
( 0.55 , 0.183 )
( 0.65 , 0.182 )
( 0.75 , 0.175 )
( 0.85 , 0.158 )
( 0.95 , 0.161 )
( 1.05 , 0.141 )
( 1.15 , 0.136 )
( 1.25 , 0.133 )
( 1.35 , 0.121 )
( 1.45 , 0.116 )
( 1.55 , 0.118 )
( 1.65 , 0.103 )
( 1.75 , 0.101 )
( 1.85 , 0.098 )
( 1.95 , 0.095 )
( 2.05 , 0.342 )
( 2.15 , 0.335 )
( 2.25 , 0.331 )
( 2.35 , 0.329 )
( 2.45 , 0.317 )
( 2.55 , 0.33 )
( 2.65 , 0.31 )
( 2.75 , 0.313 )
( 2.85 , 0.319 )
( 2.95 , 0.301 )
( 3.05 , 0.301 )
( 3.15 , 0.3 )
( 3.25 , 0.302 )
( 3.35 , 0.293 )
( 3.45 , 0.289 )
( 3.55 , 0.294 )
( 3.65 , 0.293 )
( 3.75 , 0.29 )
( 3.85 , 0.292 )
( 3.95 , 0.271 )
( 4.05 , 0.035 )
( 4.15 , 0.033 )
( 4.25 , 0.033 )
( 4.35 , 0.031 )
( 4.45 , 0.027 )
( 4.55 , 0.025 )
( 4.65 , 0.025 )
( 4.75 , 0.021 )
( 4.85 , 0.024 )
( 4.95 , 0.023 )
( 5.05 , 0.022 )
( 5.15 , 0.018 )
( 5.25 , 0.015 )
( 5.35 , 0.019 )
( 5.45 , 0.015 )
( 5.55 , 0.013 )
( 5.65 , 0.018 )
( 5.75 , 0.017 )
( 5.85 , 0.012 )
( 5.95 , 0.012 )
( 6.05 , 0.015 )
( 6.15 , 0.012 )
( 6.25 , 0.012 )
( 6.35 , 0.013 )
( 6.45 , 0.009 )
( 6.55 , 0.01 )
( 6.65 , 0.009 )
( 6.75 , 0.009 )
( 6.85 , 0.008 )
( 6.95 , 0.01 )
( 7.05 , 0.008 )
( 7.15 , 0.008 )
( 7.25 , 0.007 )
( 7.35 , 0.006 )
( 7.45 , 0.006 )
( 7.55 , 0.006 )
( 7.65 , 0.005 )
( 7.75 , 0.005 )
( 7.85 , 0.005 )
( 7.95 , 0.005 )
( 8.05 , 0.006 )
( 8.15 , 0.004 )
( 8.25 , 0.006 )
( 8.35 , 0.005 )
( 8.45 , 0.005 )
( 8.55 , 0.003 )
( 8.65 , 0.003 )
( 8.75 , 0.004 )
( 8.85 , 0.003 )
( 8.95 , 0.004 )
( 9.05 , 0.001 )
( 9.15 , 0.003 )
( 9.25 , 0.001 )
( 9.35 , 0.002 )
( 9.45 , 0.004 )
( 9.55 , 0.003 )
( 9.65 , 0.002 )
( 9.75 , 0.002 )
( 9.85 , 0.004 )
( 9.95 , 0.002 )};
        \end{axis}
    \end{tikzpicture}\caption{\footnotesize Illustration of Example \ref{ex_dw} (left) and Example \ref{ex_ns} (right). The target densities are displayed in the top images while the bottom images show the respective histograms with a sample size of \(N=50000\) each.}\label{fig1}
\end{figure}
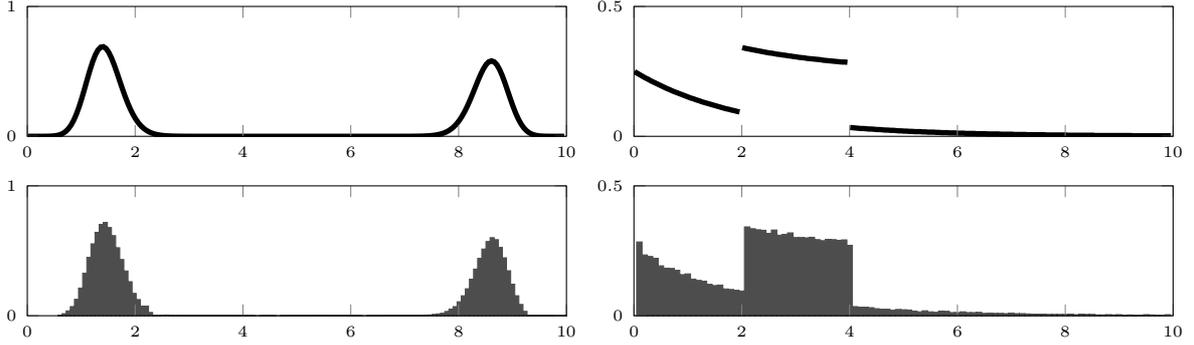

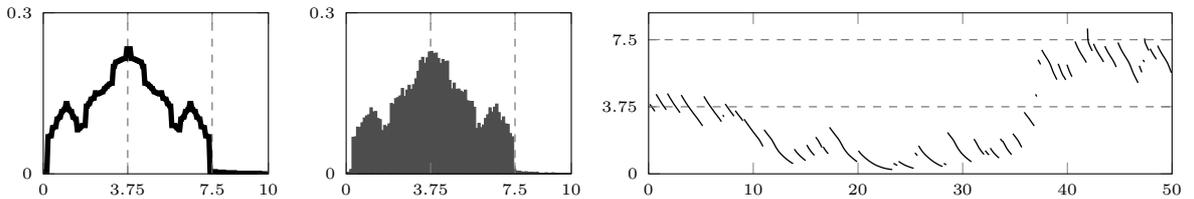
\begin{figure}
\begin{tikzpicture}
\begin{axis}[
	  height=9em,
 	  width=11em,
 	  area style,
      tick label style={font=\tiny},
 	  ytick={0,0.3,1},
 	  xtick={0,3.75,7.5,10},
	  xmin=0,
	  xmax=10,
	  ymin=0,
	  ymax=0.3]
	  \addplot[line width=0.5, color=black!50, dashed] coordinates{(7.5,0)(7.5,1)};
	  \addplot[line width=0.5, color=black!50, dashed] coordinates{(3.75,0)(3.75,1)};
		\addplot[line width=2, color=black] coordinates{( 0.01 , 0.004 )
( 0.06 , 0.004 )
( 0.11 , 0.004 )
( 0.16 , 0.004 )
( 0.21 , 0.069 )
( 0.26 , 0.069 )
( 0.31 , 0.071 )
( 0.36 , 0.071 )
( 0.41 , 0.086 )
( 0.46 , 0.086 )
( 0.51 , 0.088 )
( 0.56 , 0.088 )
( 0.61 , 0.102 )
( 0.66 , 0.102 )
( 0.71 , 0.106 )
( 0.76 , 0.106 )
( 0.81 , 0.116 )
( 0.86 , 0.116 )
( 0.91 , 0.12 )
( 0.96 , 0.12 )
( 1.01 , 0.131 )
( 1.06 , 0.131 )
( 1.11 , 0.12 )
( 1.16 , 0.12 )
( 1.21 , 0.116 )
( 1.26 , 0.116 )
( 1.31 , 0.106 )
( 1.36 , 0.106 )
( 1.41 , 0.102 )
( 1.46 , 0.102 )
( 1.51 , 0.082 )
( 1.56 , 0.082 )
( 1.61 , 0.086 )
( 1.66 , 0.086 )
( 1.71 , 0.09 )
( 1.76 , 0.09 )
( 1.81 , 0.09 )
( 1.86 , 0.09 )
( 1.91 , 0.126 )
( 1.96 , 0.126 )
( 2.01 , 0.131 )
( 2.06 , 0.131 )
( 2.11 , 0.139 )
( 2.16 , 0.139 )
( 2.21 , 0.141 )
( 2.26 , 0.141 )
( 2.31 , 0.151 )
( 2.36 , 0.151 )
( 2.41 , 0.149 )
( 2.46 , 0.149 )
( 2.51 , 0.151 )
( 2.56 , 0.151 )
( 2.61 , 0.153 )
( 2.66 , 0.153 )
( 2.71 , 0.159 )
( 2.76 , 0.159 )
( 2.81 , 0.163 )
( 2.86 , 0.163 )
( 2.91 , 0.167 )
( 2.96 , 0.167 )
( 3.01 , 0.169 )
( 3.06 , 0.169 )
( 3.11 , 0.2 )
( 3.16 , 0.2 )
( 3.21 , 0.208 )
( 3.26 , 0.208 )
( 3.31 , 0.21 )
( 3.36 , 0.21 )
( 3.41 , 0.212 )
( 3.46 , 0.212 )
( 3.51 , 0.214 )
( 3.56 , 0.214 )
( 3.61 , 0.215 )
( 3.66 , 0.215 )
( 3.71 , 0.233 )
( 3.76 , 0.233 )
( 3.81 , 0.233 )
( 3.86 , 0.233 )
( 3.91 , 0.215 )
( 3.96 , 0.215 )
( 4.01 , 0.214 )
( 4.06 , 0.214 )
( 4.11 , 0.212 )
( 4.16 , 0.212 )
( 4.21 , 0.21 )
( 4.26 , 0.21 )
( 4.31 , 0.208 )
( 4.36 , 0.208 )
( 4.41 , 0.2 )
( 4.46 , 0.2 )
( 4.51 , 0.169 )
( 4.56 , 0.169 )
( 4.61 , 0.167 )
( 4.66 , 0.167 )
( 4.71 , 0.163 )
( 4.76 , 0.163 )
( 4.81 , 0.159 )
( 4.86 , 0.159 )
( 4.91 , 0.153 )
( 4.96 , 0.153 )
( 5.01 , 0.151 )
( 5.06 , 0.151 )
( 5.11 , 0.149 )
( 5.16 , 0.149 )
( 5.21 , 0.151 )
( 5.26 , 0.151 )
( 5.31 , 0.141 )
( 5.36 , 0.141 )
( 5.41 , 0.139 )
( 5.46 , 0.139 )
( 5.51 , 0.131 )
( 5.56 , 0.131 )
( 5.61 , 0.126 )
( 5.66 , 0.126 )
( 5.71 , 0.09 )
( 5.76 , 0.09 )
( 5.81 , 0.09 )
( 5.86 , 0.09 )
( 5.91 , 0.086 )
( 5.96 , 0.086 )
( 6.01 , 0.082 )
( 6.06 , 0.082 )
( 6.11 , 0.102 )
( 6.16 , 0.102 )
( 6.21 , 0.106 )
( 6.26 , 0.106 )
( 6.31 , 0.116 )
( 6.36 , 0.116 )
( 6.41 , 0.12 )
( 6.46 , 0.12 )
( 6.51 , 0.131 )
( 6.56 , 0.131 )
( 6.61 , 0.12 )
( 6.66 , 0.12 )
( 6.71 , 0.116 )
( 6.76 , 0.116 )
( 6.81 , 0.106 )
( 6.86 , 0.106 )
( 6.91 , 0.102 )
( 6.96 , 0.102 )
( 7.01 , 0.088 )
( 7.06 , 0.088 )
( 7.11 , 0.086 )
( 7.16 , 0.086 )
( 7.21 , 0.071 )
( 7.26 , 0.071 )
( 7.31 , 0.069 )
( 7.36 , 0.069 )
( 7.41 , 0.004 )
( 7.46 , 0.004 )
( 7.51 , 0.004 )
( 7.56 , 0.004 )
( 7.61 , 0.004 )
( 7.66 , 0.004 )
( 7.71 , 0.004 )
( 7.76 , 0.003 )
( 7.81 , 0.003 )
( 7.86 , 0.003 )
( 7.91 , 0.003 )
( 7.96 , 0.003 )
( 8.01 , 0.003 )
( 8.06 , 0.003 )
( 8.11 , 0.002 )
( 8.16 , 0.002 )
( 8.21 , 0.002 )
( 8.26 , 0.002 )
( 8.31 , 0.002 )
( 8.36 , 0.002 )
( 8.41 , 0.002 )
( 8.46 , 0.002 )
( 8.51 , 0.002 )
( 8.56 , 0.002 )
( 8.61 , 0.001 )
( 8.66 , 0.001 )
( 8.71 , 0.001 )
( 8.76 , 0.001 )
( 8.81 , 0.001 )
( 8.86 , 0.001 )
( 8.91 , 0.001 )
( 8.96 , 0.001 )
( 9.01 , 0.001 )
( 9.06 , 0.001 )
( 9.11 , 0.001 )
( 9.16 , 0.001 )
( 9.21 , 0.001 )
( 9.26 , 0.001 )
( 9.31 , 0.001 )
( 9.36 , 0.001 )
( 9.41 , 0.001 )
( 9.46 , 0.001 )
( 9.51 , 0.001 )
( 9.56 , 0.001 )
( 9.61 , 0.001 )
( 9.66 , 0.001 )
( 9.71 , 0.0 )
( 9.76 , 0.0 )
( 9.81 , 0.0 )
( 9.86 , 0.0 )
( 9.91 , 0.0 )
( 9.96 , 0.0 )};
        \end{axis}
        \end{tikzpicture}
\begin{tikzpicture}
      \begin{axis}[
	  height=9em,
 	  width=11em,
 	  area style,
      tick label style={font=\tiny},
 	  ytick={0,0.3},
 	  xtick={0,3.75,7.5,10},
	  xmin=0,
	  xmax=10,
	  ymin=0,
	  ymax=0.3]
	  \addplot[line width=0.5, color=black!50, dashed] coordinates{(7.5,0)(7.5,1)};
	  \addplot[line width=0.5, color=black!50, dashed] coordinates{(3.75,0)(3.75,1)};
	  \addplot[ ybar interval,
	  fill,
                color=black!70
                ] coordinates {( 0.06 , 0.003 )
( 0.16 , 0.008 )
( 0.26 , 0.068 )
( 0.36 , 0.069 )
( 0.46 , 0.087 )
( 0.56 , 0.086 )
( 0.66 , 0.095 )
( 0.76 , 0.104 )
( 0.86 , 0.108 )
( 0.96 , 0.11 )
( 1.06 , 0.123 )
( 1.16 , 0.115 )
( 1.26 , 0.116 )
( 1.36 , 0.103 )
( 1.46 , 0.101 )
( 1.56 , 0.08 )
( 1.66 , 0.09 )
( 1.76 , 0.088 )
( 1.86 , 0.093 )
( 1.96 , 0.131 )
( 2.06 , 0.129 )
( 2.15 , 0.14 )
( 2.25 , 0.142 )
( 2.35 , 0.145 )
( 2.45 , 0.159 )
( 2.55 , 0.144 )
( 2.65 , 0.156 )
( 2.75 , 0.165 )
( 2.85 , 0.157 )
( 2.95 , 0.166 )
( 3.05 , 0.177 )
( 3.15 , 0.191 )
( 3.25 , 0.211 )
( 3.35 , 0.217 )
( 3.45 , 0.206 )
( 3.55 , 0.216 )
( 3.65 , 0.227 )
( 3.75 , 0.228 )
( 3.85 , 0.224 )
( 3.95 , 0.225 )
( 4.05 , 0.215 )
( 4.15 , 0.205 )
( 4.25 , 0.213 )
( 4.35 , 0.209 )
( 4.45 , 0.21 )
( 4.55 , 0.174 )
( 4.65 , 0.158 )
( 4.75 , 0.169 )
( 4.85 , 0.157 )
( 4.95 , 0.158 )
( 5.05 , 0.154 )
( 5.15 , 0.152 )
( 5.25 , 0.156 )
( 5.35 , 0.136 )
( 5.45 , 0.135 )
( 5.55 , 0.135 )
( 5.65 , 0.124 )
( 5.75 , 0.084 )
( 5.85 , 0.094 )
( 5.95 , 0.086 )
( 6.05 , 0.091 )
( 6.14 , 0.103 )
( 6.24 , 0.11 )
( 6.34 , 0.123 )
( 6.44 , 0.115 )
( 6.54 , 0.124 )
( 6.64 , 0.13 )
( 6.74 , 0.114 )
( 6.84 , 0.099 )
( 6.94 , 0.106 )
( 7.04 , 0.093 )
( 7.14 , 0.092 )
( 7.24 , 0.075 )
( 7.34 , 0.069 )
( 7.44 , 0.008 )
( 7.54 , 0.005 )
( 7.64 , 0.005 )
( 7.74 , 0.004 )
( 7.84 , 0.003 )
( 7.94 , 0.004 )
( 8.04 , 0.003 )
( 8.14 , 0.003 )
( 8.24 , 0.003 )
( 8.34 , 0.002 )
( 8.44 , 0.003 )
( 8.54 , 0.001 )
( 8.64 , 0.001 )
( 8.74 , 0.001 )
( 8.84 , 0.002 )
( 8.94 , 0.001 )
( 9.04 , 0.001 )
( 9.14 , 0.002 )
( 9.24 , 0.001 )
( 9.34 , 0.001 )
( 9.44 , 0.001 )
( 9.54 , 0.001 )
( 9.64 , 0.001 )
( 9.74 , 0.001 )
( 9.84 , 0.001 )
( 9.94 , 0.001 )};
        \end{axis}
    \end{tikzpicture}\begin{tikzpicture}
\begin{axis}[
	  height=9em,
 	  width=20.5em,
 	  area style,
      tick label style={font=\tiny},
 	  ytick={0,3.75,7.5},
	  xmin=0,
	  xmax=50,
	  ymin=0,
	  ymax=9]\addplot[line width=0.5, dashed, color=black!50] coordinates{(0,7.5)(50,7.5)};\addplot[line width=0.5, dashed, color=black!50] coordinates{(0,3.75)(50,3.75)};
		\addplot[line width=0.5, color=black] coordinates{
( 0.12000000000000009 , 3.8905677202104254 )
( 0.24000000000000019 , 3.7911923416675326 )
( 0.36000000000000026 , 3.6946234961629436 )
( 0.48000000000000037 , 3.5938000570420523 )
( 0.6000000000000004 , 3.495662571601529 )
};\addplot[line width=0.5, color=black] coordinates{
( 0.7200000000000005 , 4.451525670513662 )
( 0.8400000000000006 , 4.331100178475648 )
( 0.9600000000000007 , 4.214084700399095 )
( 1.0799999999999919 , 4.099312921899824 )
( 1.1999999999999786 , 3.9871961678448526 )
( 1.3199999999999654 , 3.878969323070998 )
( 1.4399999999999522 , 3.7799497709509606 )
( 1.559999999999939 , 3.6828546560300373 )
( 1.6799999999999258 , 3.5823306311861183 )
};\addplot[line width=0.5, color=black] coordinates{
( 1.7999999999999126 , 4.490578891296439 )
( 1.9199999999998993 , 4.368619764785199 )
( 2.0399999999998863 , 4.251021286820655 )
( 2.159999999999873 , 4.135547308619692 )
( 2.27999999999986 , 4.022557120137502 )
( 2.3999999999998467 , 3.9119019674890256 )
( 2.5199999999998335 , 3.810963685304802 )
( 2.6399999999998203 , 3.714114994016858 )
( 2.759999999999807 , 3.613876670455559 )
( 2.879999999999794 , 3.5151914125592367 )
( 2.9999999999997806 , 3.4189341859280513 )
};\addplot[line width=0.5, color=black] coordinates{
( 3.1199999999997674 , 4.402102694276655 )
( 3.239999999999754 , 4.283894219468085 )
( 3.359999999999741 , 4.167827614787972 )
( 3.4799999999997278 , 4.054133815055995 )
( 3.5999999999997145 , 3.942855271653416 )
( 3.7199999999997013 , 3.8388710909916783 )
( 3.839999999999688 , 3.7411088751112835 )
( 3.959999999999675 , 3.642198796650104 )
( 4.079999999999697 , 3.5427400671878466 )
( 4.199999999999737 , 3.445805028943691 )
( 4.319999999999777 , 3.3513988303583004 )
( 4.439999999999817 , 3.2597724555577323 )
( 4.5599999999998575 , 3.1702459971842885 )
( 4.679999999999898 , 3.078938352131954 )
( 4.799999999999938 , 2.9779224201424386 )
( 4.919999999999978 , 2.8783741551973097 )
( 5.040000000000018 , 2.7794075306352584 )
( 5.160000000000058 , 2.680815684324024 )
};\addplot[line width=0.5, color=black] coordinates{
( 5.280000000000098 , 4.334537513243986 )
( 5.400000000000138 , 4.217461735805532 )
( 5.520000000000178 , 4.102619191588234 )
( 5.640000000000218 , 3.9904175195852725 )
( 5.760000000000258 , 3.8818508170369324 )
( 5.880000000000298 , 3.782735464046614 )
( 6.000000000000338 , 3.6858044260128584 )
( 6.1200000000003785 , 3.5852003151862046 )
( 6.2400000000004185 , 3.487260192773002 )
( 6.360000000000459 , 3.3917602091959456 )
( 6.480000000000499 , 3.2989381922794365 )
( 6.600000000000539 , 3.2089400454495474 )
( 6.720000000000579 , 3.1195522677183285 )
( 6.840000000000619 , 3.0210978796263595 )
( 6.960000000000659 , 2.9210413598221976 )
};\addplot[line width=0.5, color=black] coordinates{
( 7.080000000000699 , 3.2841833977953425 )
( 7.200000000000739 , 3.1945858611465714 )
};\addplot[line width=0.5, color=black] coordinates{
( 7.320000000000779 , 3.903575499864665 )
( 7.440000000000819 , 3.8035461604067557 )
( 7.560000000000859 , 3.7069450626283706 )
( 7.680000000000899 , 3.606354016133165 )
( 7.8000000000009395 , 3.5078785276828612 )
( 7.92000000000098 , 3.4118075652762223 )
( 8.040000000000983 , 3.318350561379265 )
( 8.160000000000917 , 3.2277346242436864 )
};\addplot[line width=0.5, color=black] coordinates{
( 8.28000000000085 , 3.5306971632278765 )
( 8.400000000000784 , 3.4340553913038585 )
( 8.520000000000717 , 3.339980372263365 )
( 8.64000000000065 , 3.2486983401922362 )
( 8.760000000000584 , 3.159193645569718 )
( 8.880000000000518 , 3.066357355227214 )
( 9.000000000000451 , 2.965543928841139 )
};\addplot[line width=0.5, color=black] coordinates{
( 9.120000000000385 , 3.0558977522080433 )
( 9.240000000000318 , 2.9552569976713796 )
( 9.360000000000252 , 2.8558219290058044 )
( 9.480000000000185 , 2.7569843536259375 )
( 9.600000000000119 , 2.6581911490843964 )
( 9.720000000000052 , 2.5596603187850384 )
( 9.839999999999986 , 2.4630315566021026 )
( 9.95999999999992 , 2.3693694302361465 )
( 10.079999999999853 , 2.2787369124157197 )
( 10.199999999999786 , 2.1865242763986634 )
( 10.31999999999972 , 2.0965674196171786 )
( 10.439999999999653 , 2.004845107151693 )
( 10.559999999999587 , 1.9138978039016064 )
( 10.67999999999952 , 1.7943963206776374 )
( 10.799999999999454 , 1.6675546415695306 )
( 10.919999999999387 , 1.5324268645595702 )
};\addplot[line width=0.5, color=black] coordinates{
( 11.03999999999932 , 2.4590601849523477 )
( 11.159999999999254 , 2.3655952056055556 )
( 11.279999999999188 , 2.2748478078848695 )
( 11.399999999999121 , 2.18272464423085 )
( 11.519999999999055 , 2.092690170795603 )
( 11.639999999998988 , 2.0011120744244817 )
( 11.759999999998922 , 1.9101904232848281 )
( 11.879999999998855 , 1.7891854825282159 )
( 11.999999999998789 , 1.6620854524195194 )
( 12.119999999998722 , 1.5266409370045548 )
( 12.239999999998656 , 1.4106762672715534 )
( 12.359999999998589 , 1.3069812329444093 )
( 12.479999999998522 , 1.2161244660415567 )
( 12.599999999998456 , 1.1332616043434824 )
( 12.71999999999839 , 1.0597109991092912 )
( 12.839999999998323 , 0.9936335513377758 )
( 12.959999999998256 , 0.9281141306669268 )
( 13.07999999999819 , 0.8668754741755069 )
( 13.199999999998123 , 0.8100434135551139 )
( 13.319999999998057 , 0.7543607978762169 )
( 13.43999999999799 , 0.7029442034448763 )
( 13.559999999997924 , 0.6547910185652036 )
( 13.679999999997857 , 0.6113668013963884 )
( 13.799999999997791 , 0.5679735492593327 )
};\addplot[line width=0.5, color=black] coordinates{
( 13.919999999997724 , 1.3805639149831912 )
( 14.039999999997658 , 1.2802257807240058 )
( 14.159999999997591 , 1.191664694047111 )
( 14.279999999997525 , 1.1108365464551544 )
( 14.399999999997458 , 1.0405293254945704 )
( 14.519999999997392 , 0.9744695671356457 )
( 14.639999999997325 , 0.9104884647259018 )
( 14.759999999997259 , 0.8501917755949121 )
( 14.879999999997192 , 0.7943898478233391 )
( 14.999999999997126 , 0.7392378254491071 )
};\addplot[line width=0.5, color=black] coordinates{
( 15.11999999999706 , 1.6140266893901871 )
( 15.239999999996993 , 1.4804546743148264 )
( 15.359999999996926 , 1.3719525880907053 )
( 15.47999999999686 , 1.2726906744212765 )
( 15.599999999996793 , 1.1847824635794488 )
( 15.719999999996727 , 1.1043874808644485 )
( 15.83999999999666 , 1.035031544598114 )
};\addplot[line width=0.5, color=black] coordinates{
( 15.959999999996594 , 2.0530288846817686 )
( 16.07999999999667 , 1.9617689087934187 )
( 16.199999999996816 , 1.8609291162722992 )
( 16.319999999996963 , 1.7356082202742804 )
( 16.43999999999711 , 1.6058184969484073 )
( 16.559999999997256 , 1.4736912875551405 )
};\addplot[line width=0.5, color=black] coordinates{
( 16.679999999997403 , 1.5045521352089952 )
( 16.79999999999755 , 1.3936301974652292 )
( 16.919999999997696 , 1.2916911672972293 )
( 17.039999999997843 , 1.2022543477354057 )
( 17.15999999999799 , 1.1206750013525655 )
};\addplot[line width=0.5, color=black] coordinates{
( 17.279999999998136 , 2.6025531874786716 )
( 17.399999999998283 , 2.5050616490589572 )
( 17.51999999999843 , 2.409564588719415 )
( 17.639999999998576 , 2.318618940407238 )
( 17.759999999998723 , 2.226446153511319 )
( 17.87999999999887 , 2.1354810793499572 )
( 17.999999999999016 , 2.044434384262881 )
( 18.119999999999163 , 1.9532110890058048 )
( 18.23999999999931 , 1.8492468461931402 )
( 18.359999999999456 , 1.7239099030904945 )
( 18.479999999999603 , 1.5932020046701751 )
( 18.59999999999975 , 1.4632793248125067 )
( 18.719999999999896 , 1.3559001432545763 )
( 18.840000000000042 , 1.2586907451489755 )
( 18.96000000000019 , 1.1720056819122198 )
( 19.080000000000336 , 1.09298071621661 )
( 19.200000000000482 , 1.0248166476845002 )
( 19.32000000000063 , 0.9587806063478725 )
( 19.440000000000776 , 0.895965540483467 )
( 19.560000000000922 , 0.8365617572871958 )
( 19.68000000000107 , 0.7807824919238878 )
( 19.800000000001216 , 0.726882580552975 )
( 19.920000000001362 , 0.6772379306133126 )
};\addplot[line width=0.5, color=black] coordinates{
( 20.04000000000151 , 1.290009102609531 )
( 20.160000000001656 , 1.2006620348877952 )
( 20.280000000001802 , 1.1192522465790138 )
( 20.40000000000195 , 1.0477104289463817 )
( 20.520000000002096 , 0.9816516877382883 )
( 20.640000000002242 , 0.9170920277687374 )
( 20.76000000000239 , 0.8564441208663837 )
( 20.880000000002536 , 0.800543583048844 )
( 21.000000000002682 , 0.7449110747452609 )
( 21.12000000000283 , 0.6941922935274004 )
( 21.240000000002976 , 0.6467810293138447 )
( 21.360000000003122 , 0.6041881720293913 )
( 21.48000000000327 , 0.5604966290514097 )
( 21.600000000003416 , 0.5207604022293946 )
( 21.720000000003562 , 0.48500453640496294 )
( 21.84000000000371 , 0.45280261394584803 )
( 21.960000000003856 , 0.42432448223615943 )
( 22.080000000004002 , 0.39903777912248456 )
( 22.20000000000415 , 0.37245616089286915 )
( 22.320000000004296 , 0.3490896551889729 )
( 22.440000000004442 , 0.3285685319036203 )
( 22.56000000000459 , 0.3106085064808148 )
( 22.680000000004735 , 0.29476292936694787 )
( 22.800000000004882 , 0.2806640099273722 )
( 22.92000000000503 , 0.26837161274552196 )
( 23.040000000005175 , 0.25767813025087216 )
( 23.160000000005322 , 0.2484224733128777 )
( 23.28000000000547 , 0.2403974848523949 )
};\addplot[line width=0.5, color=black] coordinates{
( 23.400000000005615 , 0.5848450666021986 )
( 23.520000000005762 , 0.5425256132885921 )
( 23.64000000000591 , 0.5047317640348142 )
( 23.760000000006055 , 0.4704141961404484 )
};\addplot[line width=0.5, color=black] coordinates{
( 23.880000000006202 , 0.6969181859556188 )
( 24.00000000000635 , 0.6492458617341099 )
( 24.120000000006495 , 0.6063888395105215 )
( 24.240000000006642 , 0.5628211262446159 )
( 24.36000000000679 , 0.5228353858791858 )
( 24.480000000006935 , 0.48689434296532536 )
( 24.600000000007082 , 0.45448681015399406 )
( 24.72000000000723 , 0.4258021420310246 )
( 24.840000000007375 , 0.40050361634024306 )
( 24.960000000007522 , 0.37387241195294774 )
( 25.08000000000767 , 0.35032875256190993 )
( 25.200000000007815 , 0.32965423700682966 )
( 25.320000000007962 , 0.3115380534138804 )
};\addplot[line width=0.5, color=black] coordinates{
( 25.44000000000811 , 1.1678605984218662 )
( 25.560000000008255 , 1.0894069339602956 )
( 25.680000000008402 , 1.0215074674535511 )
};\addplot[line width=0.5, color=black] coordinates{
( 25.80000000000855 , 1.570737799052634 )
( 25.920000000008695 , 1.4452840649576413 )
( 26.040000000008842 , 1.3391422691279307 )
( 26.16000000000899 , 1.2440401628347186 )
( 26.280000000009135 , 1.1586630511971858 )
( 26.400000000009282 , 1.0815112277701846 )
( 26.52000000000943 , 1.014227667172681 )
( 26.640000000009575 , 0.9481919358462209 )
( 26.760000000009722 , 0.8859131045557939 )
( 26.88000000000987 , 0.8273931063086333 )
( 27.000000000010015 , 0.7716659707373562 )
( 27.12000000001016 , 0.7186176830775877 )
( 27.24000000001031 , 0.6694930287397421 )
( 27.360000000010455 , 0.624567809107792 )
( 27.4800000000106 , 0.5817377100295822 )
( 27.60000000001075 , 0.5397547120836548 )
( 27.720000000010895 , 0.5022484517067883 )
( 27.84000000001104 , 0.4681595951434056 )
( 27.96000000001119 , 0.4378849944109778 )
( 28.080000000011335 , 0.4111545913723171 )
};\addplot[line width=0.5, color=black] coordinates{
( 28.20000000001148 , 0.5858193471280438 )
( 28.32000000001163 , 0.5434034418718455 )
( 28.440000000011775 , 0.5055138245230337 )
};\addplot[line width=0.5, color=black] coordinates{
( 28.56000000001192 , 2.3094422221132196 )
( 28.68000000001207 , 2.216941551757327 )
( 28.800000000012215 , 2.1262176396892616 )
( 28.92000000001236 , 2.034982426264234 )
( 29.04000000001251 , 1.9438327566636913 )
( 29.160000000012655 , 1.8361167748084188 )
( 29.2800000000128 , 1.710765310442091 )
( 29.400000000012948 , 1.5786795656807802 )
( 29.520000000013095 , 1.451798177779374 )
( 29.64000000001324 , 1.3452139267656216 )
( 29.760000000013388 , 1.2493328712878213 )
( 29.880000000013535 , 1.1634618748792185 )
( 30.00000000001368 , 1.085643736984888 )
( 30.120000000013828 , 1.0180365328000864 )
( 30.240000000013975 , 0.9519947470432618 )
( 30.36000000001412 , 0.8895204696847431 )
( 30.480000000014268 , 0.8306883714053603 )
( 30.600000000014415 , 0.7749510908642383 )
( 30.72000000001456 , 0.7215922916271076 )
( 30.840000000014708 , 0.672277424807282 )
};\addplot[line width=0.5, color=black] coordinates{
( 30.960000000014855 , 0.8232053659878968 )
};\addplot[line width=0.5, color=black] coordinates{
( 31.080000000015 , 1.9455875466248755 )
( 31.200000000015148 , 1.8384451435953284 )
( 31.320000000015295 , 1.71309598404368 )
( 31.44000000001544 , 1.5812320686350465 )
( 31.560000000015588 , 1.453649495495887 )
( 31.680000000015735 , 1.34693811153216 )
( 31.80000000001588 , 1.2508374097943027 )
( 31.920000000016028 , 1.1648349507563835 )
( 32.04000000001603 , 1.0868251300638199 )
};\addplot[line width=0.5, color=black] coordinates{
( 32.16000000001575 , 1.2982115987834835 )
( 32.28000000001547 , 1.2084406296706778 )
( 32.40000000001519 , 1.126293338656513 )
( 32.52000000001491 , 1.0537183874303415 )
};\addplot[line width=0.5, color=black] coordinates{
( 32.64000000001463 , 1.4373116553577083 )
( 32.76000000001435 , 1.3317327455893035 )
( 32.88000000001407 , 1.237615399926571 )
( 33.00000000001379 , 1.1527969293404128 )
( 33.12000000001351 , 1.0764796729191992 )
( 33.24000000001323 , 1.0095851597483776 )
( 33.36000000001295 , 0.9435830293328834 )
( 33.48000000001267 , 0.8815444916424876 )
};\addplot[line width=0.5, color=black] coordinates{
( 33.60000000001239 , 1.9587354034881395 )
( 33.720000000012114 , 1.8567603962420562 )
( 33.840000000011834 , 1.8175350700886317 )
( 33.960000000011554 , 1.6918198485139035 )
( 34.080000000011275 , 1.5580750945806199 )
( 34.200000000010995 , 1.435257128958793 )
( 34.320000000010715 , 1.3298155654201071 )
( 34.440000000010436 , 1.2359306069294036 )
( 34.560000000010156 , 1.151258457614498 )
( 34.680000000009876 , 1.0751476781459453 )
};\addplot[line width=0.5, color=black] coordinates{
( 34.8000000000096 , 2.241627577520432 )
( 34.92000000000932 , 2.150295263530376 )
( 35.04000000000904 , 2.059576704434747 )
( 35.16000000000876 , 1.968262150914804 )
( 35.28000000000848 , 1.869667899809716 )
( 35.4000000000082 , 1.7443608068892416 )
( 35.52000000000792 , 1.6150002285825118 )
( 35.64000000000764 , 1.4812560846802678 )
};\addplot[line width=0.5, color=black] coordinates{
( 35.76000000000736 , 2.6350933414818396 )
};\addplot[line width=0.5, color=black] coordinates{
( 35.88000000000708 , 3.4405559240109853 )
( 36.0000000000068 , 3.3462997733099846 )
( 36.12000000000652 , 3.254826741142593 )
( 36.24000000000624 , 3.165330318375006 )
( 36.36000000000596 , 3.0733011931884042 )
( 36.48000000000568 , 2.9723786803513255 )
( 36.6000000000054 , 2.8728553452084475 )
( 36.72000000000512 , 2.7739165048004013 )
( 36.84000000000484 , 2.6752708052409084 )
};\addplot[line width=0.5, color=black] coordinates{
( 36.96000000000456 , 4.445606243299821 )
( 37.08000000000428 , 4.325430714396322 )
};\addplot[line width=0.5, color=black] coordinates{
( 37.200000000004 , 6.375768130794171 )
( 37.32000000000372 , 6.252817369370817 )
( 37.440000000003444 , 6.115535382390964 )
};\addplot[line width=0.5, color=black] coordinates{
( 37.560000000003164 , 6.929102843670829 )
( 37.680000000002885 , 6.799236857650795 )
( 37.800000000002605 , 6.68125163969978 )
( 37.920000000002325 , 6.569164276770088 )
( 38.040000000002046 , 6.461159906430887 )
( 38.160000000001766 , 6.345420046736942 )
( 38.280000000001486 , 6.218858723750407 )
( 38.40000000000121 , 6.073620421632551 )
( 38.52000000000093 , 5.888736103458611 )
( 38.64000000000065 , 5.6972064984547055 )
( 38.76000000000037 , 5.55130809253038 )
( 38.88000000000009 , 5.408054172035234 )
( 38.99999999999981 , 5.266824054092172 )
};\addplot[line width=0.5, color=black] coordinates{
( 39.11999999999953 , 6.151829994957961 )
( 39.23999999999925 , 5.98364114539602 )
( 39.35999999999897 , 5.797312813216049 )
( 39.47999999999869 , 5.625683713775754 )
( 39.59999999999841 , 5.480171875880044 )
( 39.71999999999813 , 5.389230637061151 )
( 39.83999999999785 , 5.24876445038383 )
};\addplot[line width=0.5, color=black] coordinates{
( 39.95999999999757 , 6.102588551232086 )
( 40.07999999999729 , 5.919724436078938 )
( 40.19999999999701 , 5.729653137317948 )
( 40.31999999999673 , 5.575395145726769 )
( 40.43999999999645 , 5.431634774316997 )
};\addplot[line width=0.5, color=black] coordinates{
( 40.55999999999617 , 5.870398948885487 )
};\addplot[line width=0.5, color=black] coordinates{
( 40.67999999999589 , 6.9592556389104026 )
};\addplot[line width=0.5, color=black] coordinates{
( 40.73519999999602 , 7.395738787830659 )
( 40.85519999999574 , 7.224661614705234 )
( 40.97519999999546 , 7.071620462531939 )
( 41.09519999999518 , 6.930676092454867 )
( 41.2151999999949 , 6.8007957056245365 )
( 41.33519999999462 , 6.682686817268846 )
( 41.45519999999434 , 6.570446411296561 )
( 41.57519999999406 , 6.462581445862925 )
( 41.69519999999378 , 6.346917686276397 )
( 41.8151999999935 , 6.220582831291941 )
};\addplot[line width=0.5, color=black] coordinates{
( 41.925299999993285 , 8.12451954330628 )
( 41.937299999993684 , 7.947880348224596 )
( 41.94929999999408 , 7.752497267394219 )
( 41.96129999999448 , 7.536749193538 )
( 42.02369999999456 , 7.319515079706657 )
( 42.14369999999428 , 7.154201124268122 )
( 42.263699999994 , 7.00388525532844 )
( 42.38369999999372 , 6.87180801068097 )
};\addplot[line width=0.5, color=black] coordinates{
( 42.503699999993444 , 7.266705965548134 )
( 42.623699999993164 , 7.107940140104378 )
( 42.743699999992884 , 6.962679237451094 )
( 42.863699999992605 , 6.832128608592028 )
( 42.983699999992325 , 6.710822278749295 )
( 43.103699999992045 , 6.596003840724408 )
( 43.223699999991766 , 6.490076715568122 )
( 43.343699999991486 , 6.375517555369564 )
( 43.463699999991206 , 6.25256070356818 )
};\addplot[line width=0.5, color=black] coordinates{
( 43.58369999999093 , 7.13458044073977 )
( 43.70369999999065 , 6.986000450032371 )
( 43.82369999999037 , 6.854986777367903 )
( 43.94369999999009 , 6.731776254424209 )
( 44.06369999998981 , 6.615935083953479 )
( 44.18369999998953 , 6.509294663042658 )
( 44.30369999998925 , 6.396309478603101 )
( 44.42369999998897 , 6.2756970160982934 )
( 44.54369999998869 , 6.140233405163378 )
( 44.66369999998841 , 5.968602187664094 )
};\addplot[line width=0.5, color=black] coordinates{
( 44.78369999998813 , 6.47910031139827 )
};\addplot[line width=0.5, color=black] coordinates{
( 44.83349999998829 , 7.344194670584592 )
( 44.95349999998801 , 7.175432178012719 )
( 45.07349999998773 , 7.025558099625302 )
( 45.19349999998745 , 6.890163585101335 )
( 45.31349999998717 , 6.7642999907354096 )
( 45.43349999998689 , 6.6474011682073755 )
( 45.55349999998661 , 6.538073315152533 )
( 45.67349999998633 , 6.4276166478432595 )
( 45.793499999986054 , 6.3104556940542125 )
( 45.913499999985774 , 6.178744666519565 )
( 46.033499999985494 , 6.019137633165334 )
( 46.153499999985215 , 5.833807134879324 )
( 46.273499999984935 , 5.653877031443621 )
( 46.393499999984655 , 5.507462809351536 )
( 46.513499999984376 , 5.365622228647399 )
( 46.633499999984096 , 5.2259891377077174 )
( 46.753499999983816 , 5.084984763358073 )
};\addplot[line width=0.5, color=black] coordinates{
( 46.87349999998354 , 6.043036497625936 )
( 46.99349999998326 , 5.8579328377648965 )
( 47.11349999998298 , 5.672949462064056 )
};\addplot[line width=0.5, color=black] coordinates{
( 47.2334999999827 , 6.374134133884101 )
( 47.35349999998242 , 6.251057260402117 )
};\addplot[line width=0.5, color=black] coordinates{
( 47.37629999998275 , 7.5721852141271375 )
( 47.42249999998293 , 7.346169758400921 )
( 47.54249999998265 , 7.177208498478677 )
( 47.66249999998237 , 7.027360208890025 )
( 47.78249999998209 , 6.891738522922698 )
};\addplot[line width=0.5, color=black] coordinates{
( 47.90249999998181 , 6.997504097576161 )
( 48.022499999981534 , 6.866240827490978 )
( 48.142499999981254 , 6.74219142728606 )
( 48.262499999980975 , 6.626026897437707 )
( 48.382499999980695 , 6.518484660155965 )
( 48.502499999980415 , 6.406421408233817 )
};\addplot[line width=0.5, color=black] coordinates{
( 48.622499999980135 , 7.170525703521194 )
( 48.742499999979856 , 7.020548364470858 )
( 48.862499999979576 , 6.885927217592494 )
( 48.9824999999793 , 6.760390036403388 )
( 49.10249999997902 , 6.643617756800794 )
( 49.22249999997874 , 6.53461136019164 )
( 49.34249999997846 , 6.423865729355782 )
( 49.46249999997818 , 6.306537775931335 )
( 49.5824999999779 , 6.174133189732823 )
( 49.70249999997762 , 6.012942422590834 )
( 49.82249999997734 , 5.827568493694143 )
( 49.94249999997706 , 5.649048255471964 )
};
        \end{axis}
        \end{tikzpicture}\caption{From left to right we see the target density function \(\pi\), a histogram of the sampled distribution with sample size \(N=50000\), and an exemplary sample path of \(\X\).}\label{fig2}
\end{figure}

\begin{example}[double-well]\label{ex_dw}
In \cite{FLI} is is pointed out that sampling from a target distribution \(\ppi(\diff x)=\pi(x)\diff x\) with two separated modes is challenging for classic \textsc{lmc}. The lower the values of \(\pi\) are between the modes the longer it takes on average for the continuous Langevin diffusion to move from one mode to the other.
This issue can be circumvented by allowing jumps.
Take
\begin{align*}
\pi(x)&=\exp\left\{\frac1{10}x(x-4)(x-6.02)(x-10)+0.5\right\}, \quad x>0,
\end{align*}
which is taken from \cite[Sec. 4]{FLI}, but shifted to the right such that both modes are contained in \((0,\infty)\). As driving noise we choose a Lévy process \(\L\) with characteristic triplet \((0,0,\mu)\) where
\begin{align*}
\mu(\diff x)&=\ee^{-x}\diff x + \delta_4 + 2\delta_8.
\end{align*}
Clearly, conditions \textbf{(b1)}, \textbf{(b2)}, and \textbf{(c2)} are fulfilled. Thus, \(\ppi\) is invariant for a solution \(\X\) of \eqref{eq_sde} with \(\phi\) as in \eqref{eq_driftcpn} by Theorem \ref{thm_cpn}. Further, since Lemma \ref{lem_full} applies, \(\ppi\) is even limiting.\\
We demonstrate this example in Figure \ref{fig1}.\\
 Note that, in general, there is no closed-form expression of \(\phi\) due to the convolution term appearing in its definition. Hence, we use numerical integration to evaluate \(\phi\) for this and the following two examples.
\end{example}

\begin{example}[non-smooth density]\label{ex_ns}
Another important advantage of \textsc{llmc} compared to \textsc{lmc} and also f\textsc{lmc} of \cite{FLI} is the possibility to choose non-smooth target densities. Let, for example, \(\ppi(\diff x)=\pi(x)\diff x\) with
\begin{align*}
\pi(x)&=\ee^{-0.5x}+\bone_{(2,4)}(x), 
\end{align*}
and \(\L\) be a Lévy process with characteristic triplet \((0,0,\mu)\) where
\begin{align*}
\mu(\diff x)&=x^2\ee^{-0.5x}\diff x + \delta_1.
\end{align*}
Let \(\X\) be a solution of \eqref{eq_sde} with \(\phi\) as in \eqref{eq_driftcpn}. As for the previous example, \textbf{(b1)}, \textbf{(b2)}, and \textbf{(c2)} are met, and \(\ppi\) is invariant and limiting for \(\X\) by Theorem \ref{thm_cpn} and Lemma \ref{lem_full}.\\
This, too, is displayed in Figure \ref{fig1}.
\end{example}

\begin{example}[\textsc{Dresden Frauenkirche}]
To illustrate that our result also covers target densities with lots of detail we consider the density \(\pi\) as in Figure 2 (left) which represents the silhouette of the \textsc{Dresden Frauenkirche}, continued by an exponential tail. We manufactured \(\pi\) in a way such that \textbf{(b1)} is met. As driving noise we choose a spectrally positive Lévy process with characteristic triplet \((0,0,\mu)\) where
\begin{align*}
\mu(\diff x)=\bone_{\{x>0\}}\ee^{-\frac{x^2}2}\diff x.
\end{align*}
Let \(\X\) be a solution of \eqref{eq_sde} with \(\phi\) as in \eqref{eq_driftcpn}. As for both prior examples, \(\ppi\) is clearly invariant and limiting for \(\X\). The density function, the sampled distribution and an exemplary sample path of \(\X\) can be seen in Figure \ref{fig2}.\\
Taking a closer look we see that the process \(\X\) slows down considerably upon entering the interval \((0,7.5)\) on which most of the mass of \(\ppi\) concentrates. In general, the drift coefficient takes on large values in areas of small mass and small values in areas of high mass. This stems from \(\pi\) appearing in the denominator of  \eqref{eq_driftcpn}, and can be observed by inspecting the slopes of the sample path in Figure 2.\\
Moreover, the process becomes slower the closer it gets to the origin. On the one hand, this is due to Assumption \textbf{(b1)} (and \textbf{(a2)}, respectively), and ensures that \(0\) cannot be reached in finite time. On the other hand, this slowing down is caused by the convolution with the signed tail function \(\overline\mu_s\) in the nominator of \eqref{eq_driftcpn}.\\
This is reasonable: Because jumps go only upwards it is \emph{less likely} for \(\X\) to reach the area of the left half of the silhouette (approximately the interval \((0,3.75)\)) than to appear in the area of its right half. But since both sides are symmetrical the drift must compensate for that.
\end{example}

\section{Proof of Theorem \ref{thm_cpn}}\label{sec_prf}

n the following, whenever constants \(C,C'\) or \(\alpha\) appear in the proof below, we mean the constants of Assumption \textbf{(b1)}.

\subsubsection*{Proof of Theorem \ref{thm_cpn} (i): Positive Harris recurrence}
The Foster-Lyapunov method of \cite{MTI} - \cite{MTIII} is tailored for processes which cover the whole real line. Hence, we consider the auxiliary process \(Y_t=s(X_t)\) where \(s:(0,\infty)\to\real\) is a smooth strictly monotone function such that \(s(x)=\ln(x)\) for \(x\in(0,1-\varepsilon)\) and \(s(x)=x\) for \(x\in(1+\varepsilon,\infty)\) where \(0<\varepsilon<\ee^{-1}\) is some constant. Clearly, \(\X\) is positive Harris recurrent if and only if \(\Y\) is positive Harris recurrent.\\
Central to this method are the so-called \emph{norm-like functions} whose precise definition needs additional notation: For \(m\in\nat\) denote \(O_m:=(\ee^{-m},m)\) and choose \(h_m\in\testo\) such that \(0\leq h_m\leq 1\), and with \(h_m(x)=1\) for all \(x\in O_m\), and \(h_m(x)=0\) for all \(x\in O_{m+1}^c\). Let \(\Xmm\) be the unique strong solution of
\begin{align*}
\diff \mathcal{X}^{(m)}_t = h_m(\mathcal{X}^{(m)}_{t-})\left(\phi(\mathcal{X}^{(m)}_{t-})\diff t + \diff L_t\right),\quad \mathcal{X}^{(m)}_0>0.
\end{align*}
We set \(\mathcal{Y}^{(m)}_t:=s(\mathcal{X}^{(m)}_t)\). Clearly, this construction implies that for all \(m\in\nat\) if \(\mathcal{X}^{m}_0\in O_m\) it holds \(\mathcal{Y}^{(m)}_t = Y_t\) for all \(t< T_m:=\inf\{s\geq0: |Y_s|\geq m\}\). For \(m\in\nat\) denote by \((\cG_m,\cD(\cG_m))\) the extended generator of \(\Ymm\). A function \(f:\real\to\real_+\) is called \emph{norm-like} w.r.t. \(\Y\) if
\begin{enumerate}
\item \(f(x)\to\infty\) as \(x\to\pm\infty\), and
\item \(f\in\cD(\cG_m)\) for all \(m\in\nat\).
\end{enumerate}
It is typical for the Foster-Lyapunov method that one only requires a single norm-like function (sometimes also called Foster-Lyapunov function) which fulfills a certain inequality. Our particular choice is presented in the lemma below. \\
In the following, to make notation easier, we denote \(y:=s(x)\) if \(x\in(0,\infty)\) is given, and \(x:=s^{-1}(y)\) if \(y\in\real\) is given.
\begin{lemma}\label{lem_norm}
Let \(f\in\cC^1(\real)\) with \(f(y)\in[1+|y|,2+|y|]\) for all \(y\in\real\), and \(f(y)=1+|y|\) for all \(|y|>\varepsilon\) for some \(\varepsilon>0\). Then \(f\) is norm-like w.r.t. \(\Y\).\\
Moreover, for all \(m\in\nat\) and \(y\in(-m,m)\) it holds
\begin{align}\label{eq_Gmf}
\cG_m f(y)=\phi(x) f_0'(x) + \int_{(0,\infty)}(f_0(x+z)-f_0(x)) \mu(\diff z).
\end{align}
where \(f_0(x):=f(y)=f(s(x))\).
\end{lemma}
\begin{proof}
Fix \(m\in\nat\). Itô's formula yields
\begin{align}\label{eq_genito}\nonumber
\EE^y[f(\mathcal{Y}^{(m)}_t)-f(y)] &= \EE^{x}[f_0(\mathcal{X}^{(m)}_t)-f_0(x)]\\\nonumber
&=\EE^{x}\Big[\int_0^t h_m(\mathcal{X}^{(m)}_{s-}) \phi(\mathcal{X}^{(m)}_{s-})f_0'(\mathcal{X}^{(m)}_{s-})\diff s  \\
&\quad +\int_{z\neq0}\int_0^t \left(f_0(\mathcal{X}^{(m)}_{s-}+h_m(\mathcal{X}^{(m)}_{s-})z)- f_0(\mathcal{X}^{(m)}_{s-})\right)\widetilde\mu(\cdot,\diff s,\diff z)\Big]
\end{align}
where \(\widetilde\mu\) is the jump measure of \(\L\). To verify whether the jump measure may be replaced by the compensator under the expectation, and to subsequently swap the order of integration, we need some estimates. Clearly, \(f_0(x+h_m(x)z)-f_0(x)=0\) for all \(z>0\), and \(x\notin[\ee^{-m-1},m+1]\). On the other hand, for all \(x\in[\ee^{-(m+1)},m+1]\) there exists \(M>0\) such that
\begin{align*}
|f_0(x+h_m(x)z)-f_0(x)| \leq M\vee z
\end{align*}
by the definition of \(f_0\). Hence, by \cite[Thm. 2.21]{schnurr2009symbol} and the fact that \(\int_0^\infty (1\vee z) \mu(\diff z)<\infty\) by Assumption \textbf{(b2)}, \(\widetilde\mu~(\cdot,\diff s,\diff z)\) may be replaced by \(\diff s\mu(\diff z)\) under the expectation in \eqref{eq_genito}. \\
Applying Fubini's theorem and reversing the space transform, i.e. going back to \(\Ym\), we obtain
\begin{align*}
\EE^y[f(\mathcal{Y}^{(m)}_t)-f(y)] &=\EE^{x}\Big[\int_0^t h_m(\mathcal{X}^{(m)}_{s-}) \phi(\mathcal{X}^{(m)}_{s-})f_0'(\mathcal{X}_{s-})\diff s  \\
&\quad +\int_0^t \int_{z\neq0}\left(f_0(\mathcal{X}^{(m)}_{s-}+h_m(\mathcal{X}^{(m)}_{s-})z)- f_0(\mathcal{X}^{(m)}_{s-})\right)\mu(\diff z)\diff s\Big]\\
&=\EE^{y} \left[\int_0^tg(\mathcal{Y}^{(m)}_{s-})\diff s\right]
\end{align*}
where 
\begin{align}\label{eq_Gmf2}\nonumber
g(y)&:= h_m(x)\phi(x) f_0'(x)\\
&\qquad + \int_{z\neq0} (f_0(x+h_m(x)z)-f_0(x)) \mu(\diff z).
\end{align}
This function is clearly measurable. We observe that the integral term is continuous in \(y\) and vanishes for \(|y|\geq m+1\). Therefore, \(g\) is bounded and Tonelli's theorem is applicable yielding
\begin{align*}
\left|\EE^{y} \left[\int_0^tg\left(\mathcal{Y}^{(m)}_{s-}\right)\diff s\right]\right|\leq\int_0^t\EE^{y} \left[\left|g\left(\mathcal{Y}^{(m)}_{s-}\right)\right|\right]\diff s \leq \left\|g\right\|_\infty t <\infty
\end{align*}
for all \(y\in\real\) and all \(t\geq0\). Hence, \(f\in\cD(\cG_m)\) for all \(m\in\nat\). This completes the proof as we observe that \eqref{eq_Gmf} follows from the definition of the extended generator and upon realizing that the representation in \eqref{eq_Gmf2} agrees with \eqref{eq_Gmf} for all \(y\in(-m,m)\).
\end{proof}
The second key ingredient of the Foster-Lyapunov method is the following: A set \(K\subset\real\) is called \emph{petite} for a Markov process \((Y_t)_{t\geq0}\) if there exists a distribution \(a\) on \((0,\infty)\) and a non-trivial measure \(\varphi\) on \(\cB(\real)\) such that for all \(y\in K\) and \(B\in\cB(\real)\)
\begin{align*}
\int_{(0,\infty)} \PP^y(Y_t\in B)a(\diff t) \geq \varphi(B).
\end{align*}
\begin{lemma}\label{lem_pet}
All compact sets \(K\subset\real\) are petite for \(\Y\).
\end{lemma}

\begin{proof}
We start with some helpful notation. For \(y\in\real\) denote by \(q_y(\cdot)\) the solution of the autonomous differential equation
\begin{align}\label{eq_auto}
\begin{cases}
q_y'=\phi(s^{-1}(q_y))s'(s^{-1}(q_y)),\\
q_y(0)=y.
\end{cases}
\end{align}
Then \(q_y(t)\) represents the (deterministic) state \(Y_t^{y}\) under the assumption that no jump occurs in the time interval \([0,t]\), that is \(N_t=0\). The inverse function \(q_y^{-1}(y')\) exists due to \(\phi(x)<0\) for all \(x>0\). We note that it represents the time it takes to drift from \(y>0\) to \(y'\in(0,y)\), and is hence decreasing in \(y'\) and increasing in \(y\). \\
Let \(K\subset\real\) be compact, without loss of generality assume \(K=[k_1,k_2]\). Let \(a(\diff t)=\ee^{-t}\diff t\) and \(\varphi(\diff z)=c\bone_{(k_1-1,k_1)}(z)\diff z\) for some \(c>0\) which we are yet to choose. Let \(y\in K\) and \(B\in\cB(\real)\). Using \(\PP(N_t=0)=\ee^{-t|\mu|}\) we compute
\begin{align*}
\int_0^\infty \PP^y(Y_t\in B) a(\diff t)
&\geq \int_0^\infty \bone_{q_y^{-1}(B)}(t) \ee^{-t(1+|\mu|)}\diff t \\
&\geq\int_0^{q^{-1}_{y}(k_1-1)} \bone_{q_y^{-1}(B)}(t) \ee^{-t(1+|\mu|)}\frac{q_y'(t)}{q_y'(t)}\diff t \\
&\geq\frac1{\sup\{|q_y'(t)|: t\leq q^{-1}_{y}(k_1-1)\}}\int_{k_1-1}^y \bone_B(z)\ee^{-q_y^{-1}(z)(1+|\mu|)}\diff z\\
&\geq\frac1{\sup\{|q_y'(t)|: t\leq q^{-1}_{y}(k_1-1)\}}\int_{k_1-1}^{k_1} \bone_B(z)\ee^{-q_y^{-1}(z)(1+|\mu|)}\diff z\\
&\geq \frac{\exp\{-q^{-1}_{k_2}(k_1-1)(1+|\mu|)\}}{\sup\{|q_y'(t)|: t\leq q^{-1}_{y}(k_1-1)\}}\int_{k_1-1}^{k_1} \bone_B(z)\diff z.
\end{align*}
For the third inequality we substituted \(z:=q_y(t)\) and used the fact that for all \(y,k_1\in\real\) it holds \(\sup\{|q_y'(t)|: t\leq q^{-1}_{y}(k_1-1)\}<\infty\). Indeed, this is implied by \eqref{eq_auto}, and the properties of \(\phi\) and \(s\). The fourth inequality is due to the reduction of the area of integration while the fifth inequality uses the monotonicity properties of \(q^{-1}\) described above. Lastly, choosing 
\begin{align*}
c:=\frac{\exp\{-q^{-1}_{k_2}(k_1-1)(1+|\mu|)\}}{\sup\{|q_y'(t)|~:~t~\leq~ q^{-1}_{y}(k_1-1)\}}
\end{align*}
 finishes the proof.
\end{proof}

Finally, we are ready to prove the first claim of Theorem \ref{thm_cpn}.
\begin{proof}[Proof of Theorem \ref{thm_cpn} (i)]
We show that there exist some positive constants \(c,d>0\) and a closed petite set \(K\subset \real\) such that
\begin{align}\label{eq_fle}
\cG_mf(y)\leq -c + d\bone_K (y)
\end{align}
for all  \(m\in\nat\) and \(y\in(-m,m)\). Then \cite[Thm. 4.2]{MTIII} implies that \(\Y\) is positive Harris recurrent, and therefore, \(\X\) is positive Harris recurrent as well.\\ Clearly, the function
\begin{align}\label{eq_genfun}
\cG_m f: y\mapsto \phi(x)f_0'(x) + \int_{(0,\infty)}(f_0(x+z)-f_0(x))\mu(\diff z)
\end{align}
is continuous, and bounded on \((-m,m)\). \\
Hence, \eqref{eq_fle} follows if we can show that \(\limsup_{y\to\pm\infty} \cG_mf(y)\leq -c\) for some \(c>0\). We start with \(y\to+\infty\). Note that for \(y\gg1\) we have \(x=s^{-1}(y)=y\), and, on the one hand \(f_0'(y)=1\), and, on the other hand \(f_0(y+z)-f_0(y)=z\).\\
With Assumption \textbf{(b1)} we obtain
\begin{align*}
\limsup_{y\to+\infty} \phi(y)=-\liminf_{y\to+\infty}(\overline\mu_s*\pi(y))\frac{\ee^{\alpha y}}C \leq -\liminf_{y\to+\infty} \frac1C\int_{(0,M)}\overline\mu_s(z)\pi(y-z)\ee^{\alpha y}\diff z
\end{align*}
for all arbitrary, but fixed \(M>0\). Thus, also with Assumption \textbf{(b1)}, 
\begin{align*}
\limsup_{y\to+\infty} \phi(y)f_0'(y) &\leq -\liminf_{y\to+\infty} \int_{(0,M)}\overline\mu_s(z)\ee^{\alpha z}\diff z = -\int_{(0,M)}\overline\mu_s(z)\ee^{\alpha z}\diff z
\end{align*}
Since \(M>0\) was arbitrary this yields \(\limsup_{y\to+\infty} \phi(y)f_0'(y)< \int_{(0,\infty)}\overline\mu_s(z)\diff z = - \EE \xi_1\). Now, for the second term of \eqref{eq_genfun} we observe that for \(x=y\gg1\)
\begin{align*}
\int_{(0,\infty)}(f_0(x+z)-f_0(x))\mu(\diff z)&=
\int_{(0,\infty)}z \mu(\diff z) =\EE\xi_1.
\end{align*}
Consequently, there exists \(c>0\) such that \(\cG_m f(y)< -c <0\) for \(y\gg1\).\\
Next, consider the behavior for \(y\to-\infty\), and start with the observation that for \(y\ll-1\) one has \(x=s^{-1}(y)=\ee^y\), and \(f_0'(x)=\ee^{-y}\).
With the definition of \(\phi\) and Assumption \textbf{(b1)} we obtain 
\begin{align*}
|\phi(\ee^y)|\leq C'|\mu|\ee^y
\end{align*}
for \(y\ll-1\). Therefore, \(\phi(\ee^y)f_0'(\ee^y)\) is bounded for \(y\ll-1\).\\
Finally, to find a suitable estimate for the second term of \eqref{eq_genfun} for \(y\ll-1\) we fix \(M>0\) such that \(\mu([M,\infty))>0\). Observe that for \(y\ll-1\) it holds \(f_0(\ee^y+z)-f_0(\ee^y)<0\) for all \(z\in(0,M)\). Further, there exists \(K>0\) such that \(f_0(\ee^y+z)<M'+z\) for all \(z\in[M,\infty)\). We then compute
\begin{align*}
\int_{(0,\infty)}(f_0(\ee^y+z)-f_0(\ee^y)) \mu(\diff z)&=  \int_{(0,M)}(f_0(\ee^y+z)-f_0(\ee^y)) \mu(\diff z)\\
&\quad +\int_{[M,\infty)}(f_0(\ee^y+z)-f_0(\ee^y)) \mu(\diff z)\\
&\leq \int_{[M,\infty)}(K+z)\mu(\diff z) - f_0(\ee^y)\mu[M,\infty).
\end{align*}
As \(\int_{(0,\infty)}z\mu(\diff z)<\infty\) this implies
\begin{align*}
\lim_{y\to-\infty}\int_{(0,\infty)}(f_0(\ee^y+z)-f_0(\ee^y) \mu(\diff z)=-\infty,
\end{align*}
and therefore, \(\cG_m f(y)< -c\) for \(y\ll-1\) with the same \(c\) as above. This completes the proof.
\end{proof}

\subsubsection*{Proof of Theorem \ref{thm_cpn} (ii): Invariant distributions are infinitesimally invariant}

\begin{proof}[Proof of Theorem \ref{thm_cpn} (ii)]
The claim follows from \cite[Cor. 5.4]{behmeoechsler} if we can show that \(\frac1t\left|\EE^xf(X_t)-f(x)\right|<\infty\) for all \(f\in\cC_c^\infty(0,\infty)\) and all \(t\geq0\).\\
Analogously to the proof of Lemma \ref{lem_norm}, a straight-forward application of Itô's formula yields
\begin{align*}
\frac1t\left|\EE^xf(X_t)-f(x)\right| = \frac1t \left|\EE^x \int_0^t g(X_{s-})\diff s\right|
\end{align*}
for \(f\in\cC_c^\infty(0,\infty)\), where
\begin{align*}
g(x)=\phi(x)f'(x)+\int_{(0,\infty)}(f(x+z)-f(x))\mu(\diff z).
\end{align*}
Clearly, if \(f\in\cC_c^\infty(0,\infty)\), then \(g\) is bounded and it follows
\begin{align*}
\frac1t\left|\EE^xf(X_t)-f(x)\right|\leq \|g\|_\infty <\infty.
\end{align*}
\end{proof}

\subsubsection*{Proof of Theorem \ref{thm_cpn} (iii): Uniqueness of the invariant distribution}

For the third assertion we require one of the additional assumptions. As described above we start with Assumption \textbf{(c1)}, i.e. there exists \(n\in\nat\) such that \(\supp\mu\subset(1/n,n)\).
\begin{proof}[Proof of Theorem \ref{thm_cpn} (iii) under \textbf{(c1)}]
It has been shown in \cite[Thm. 4.2]{behmeoechsler} that any infinitesimally invariant measure \(\eeta\) of a solution \(\X\) of \eqref{eq_sde} necessarily solves the distributional equation
\begin{align}\label{eq_1}
-(\phi \eeta) ' + \mu * \eeta -|\mu|\eeta=0
\end{align}
on \((0,\infty)\). To show that there exists only one probability distribution solving \eqref{eq_1} we first need some regularity properties for \(\phi\). A straight-forward calculation yields that for all \(x\geq0\) the representation
\begin{align}\label{eq_phi1}
\phi(x)=\frac{\int_0^x (\mu*\pi(z)-|\mu|\pi(z))\diff z}{\pi(x)}
\end{align}
holds. As \(\pi\) is piecewise weakly differentiable and \(\pi(x)>0\) for \(x>0\) it follows that \(1/\pi\) is piecewise weakly differentiable as well. Thus, \(\phi\) is piecewise weakly differentiable w.r.t. the same partition as \(\pi\), since the numerator of the right-hand side of \eqref{eq_phi1} is the primitive of a locally integrable function, and as such contained in \(\sobolev\). Further, as \(\phi(x)<0\) for all \(x>0\) we infer that at least \(1/\phi\in L^1_\loc(0,\infty)\).\\
This property of \(\phi\) allows us to transform \eqref{eq_1} into
\begin{align}\label{eq_2}
\eeta'=\frac{\mu*\eeta-\phi'\eeta-|\mu|\eeta}{\phi}.
\end{align}
We note that the right-hand side of \eqref{eq_2} defines a Schwartz distribution (cf. \cite[Lem. 2.2]{behmeoechsler}) if \(\eeta\) is a real-valued Radon measure which we can assume as we are only looking for solutions which are probability distributions. More importantly, in this case the right-hand side of \eqref{eq_2} is even a Schwartz distribution of order \(0\), i.e. it can be identified with some real-valued Radon measure on \((0,\infty)\).\\
In summary, the distributional derivative of \(\eeta\) can be identified with a real-valued Radon measure which implies that \(\eeta\) itself can be identified with a locally integrable function. But now, if we insert a locally integrable function \(\eeta\) into the right-hand side of \eqref{eq_2} we obtain a locally integrable function plus a discrete measure with atoms at the discontinuities of \(\phi\). Integrating on both sides of \eqref{eq_2} tells us that any solution of \eqref{eq_1} is piecewise weakly differentiable w.r.t. the same partition as \(\pi\).\\
Let \(n\in\nat\) such that \(\supp \mu\subset(1/n,n)\). To solve \eqref{eq_1} we integrate both sides and obtain
\begin{align}\label{eq_3}
-\phi\eeta + \overline\mu_s*\eeta=c_1
\end{align}
for some \(c_1\in\real\). Observe that for \(x\gg1\) we have
\begin{align*}
|\phi(x)| \leq |\mu|\int_{1/n}^n\frac{ \pi(x-z)}{\pi(x)}\diff z\leq |\mu|\ee^{\alpha n},
\end{align*}
by Assumption \textbf{(b1)}. Thus, \(\limsup_{x\to\infty}|\phi(x)|<\infty\). From Young's convolution inequality it follows that \(\|\overline\mu_s*\eeta\|_1\leq \|\overline\mu_s\|_1\|\eeta\|_1<\infty\). Thus, for any absolutely continuous measure \(\eeta\) the left-hand side of \eqref{eq_3} can be identified with an element of \(L^1(\real_+)\). Consequently, \(c_1=0\).\\
Above we have seen that any probability distribution \(\eeta\) solving \eqref{eq_1} is absolutely continuous. Denoting by \(H(x):=\eeta((0,x])\) the cumulative distribution function of \(\eeta\) we thus know that the density function \(H'\) of \(\eeta\) is integrable. If we once again use the fact that \(\int_0^x (\mu*H'(z)-|\mu|H'(z))\diff z = \overline\mu_s*(H')(x)\), and that \(\mu*(H')(x)=0\) for \(x\in(0,1/n]\), we obtain from \eqref{eq_3} the equation
\begin{align}\label{eq_4}
H'(x) = -\frac{|\mu|H(x)}{\phi(x)}, \qquad x\in(0,1/n].
\end{align}
 Caratheodory's theorem (cf. \cite[Thm. 5.3]{hale}) implies that \eqref{eq_4} has for each initial value \(H(\varepsilon)=c_2\in\real\) a unique solution. Clearly, for arbitrary but fixed \(c_2\in\real\) the solution of \eqref{eq_4} is given by \(H(x)=\frac{c_2}{F(\varepsilon)}F(x)\) where \(F(x):=\ppi(0,x]\) is the cumulative distribution function of our target distribution \(\ppi\). This is easily seen with the definition of \(\phi\). \\
For \(x>1/n\) Equation \eqref{eq_3} reads
\begin{align}\label{eq_5}
H'(x)=\frac{\int_0^x\mu*(H')(z)\diff z-|\mu|H(x)}{\phi(x)}.
\end{align}
From Assumption \textbf{(c1)} it follows that for all \(0<x<b\) and any function \(f\in L^1(\real_+)\) holds 
\begin{align*}
\mu*f(x)=\mu*(f|_{\left[0,b-\frac1n\right]})(x).
\end{align*}
Consider Equation \eqref{eq_5} on \([m/n,(m+1)/n]\) for some \(m\in\nat\). As initial condition we assume \(H(x)=c_3F(x)\) for all \(x\in(0,m/n]\) and some \(c_3\in\real\). This results in the equation
\begin{align*}
H'(x)=\frac{c_3\int_0^x\mu*\pi(z)\diff z-|\mu|H(x)}{\phi(x)}, \qquad x\in[m/n,(m+1)/n]
\end{align*}
for which Caratheodory's theorem again ensures a unique solution. Hence, by induction over \(m\) and subsequent normalization it follows that \(\ppi\) is the unique infinitesimally invariant distribution of \(\X\).\\
Finally, Theorem \ref{thm_cpn} (i), that is positive Harris recurrence, implies existence and uniqueness of an invariant distribution (cf. \cite[Sec. 4]{MTIII}). But this unique distribution must be \(\ppi\) due to Theorem \ref{thm_cpn} (ii). This proves the claim.
\end{proof}

\begin{proof}[Proof of Theorem \ref{thm_cpn} (iii) under \textbf{(c2)}]
Assume \(\phi\) is piecewise locally Lipschitz continuous. In this case we show the assertion by approximating \(\X\) with a sequence of processes meeting Assumption \textbf{(c1)}. \\
Recall that \(L_t:= \sum_{i=0}^{N_t}\xi_i\), and define
\begin{align*}
L^{(n)}_t&:=\sum_{i=0}^{N_t}\xi_i\bone_{\left\{\frac1n<\xi_i<n\right\}}
\end{align*}
for all \(n\in\nat\). Evidently, \(\Ln\) is a Lévy process with characteristic triplet \((0,0,\mu^{(n)})\) where \(\mu^{(n)}(B)=\mu(B\cap(\frac1n,n))\) for all \(B\in\cB(0,\infty)\). To avoid the trivial case we consider only \(n\in\nat\) large enough such that \(\supp \mu\cap (\frac1n,n)\neq\emptyset\).\\
Further denote 
\begin{align*}
\phi_n(x)=-\frac{\overline\mu_s^{(n)}*\pi(x)}{\pi(x)}.
\end{align*}
Observe that for \(n\in\nat\) large enough and all \(x\in(0,\infty)\) it holds
\begin{align*}
|\phi(x)-\phi_n(x)|=\frac{(\overline\mu_s-\overline\mu_s^{(n)})*\pi(x)}{\pi(x)} \leq \mu((0,1/n]\cup[n,\infty)) \frac{\int_0^x \pi(z)\diff z}{\pi(x)}.
\end{align*}
 With \(\pi\) being a probability density the numerator is bounded. Using Assumption \textbf{(b1)} and the fact that \(\pi\) is bounded away from zero on compact intervals in \((0,\infty)\) reveals that for all compact sets \(K\subset[0,\infty)\) there exists \(c>0\) such that
\begin{align}\label{eq_estimatephi}
|\phi(x)-\phi_n(x)| \leq c \mu((0,1/n]\cup[n,\infty)) 
\end{align}
for all \(x\in K\). \\
For \(n\in\nat\) large enough we consider now solutions \(\Xn\) of the stochastic differential equations
\begin{align}\label{eq_sde_aux}
\diff X^{(n)}_t = \phi_n(X^{(n)}_{t-})\diff t + L^{(n)}_t, \quad X^{(n)}_0\sim \ppi.
\end{align}
Note that, by construction, the processes \(\X\) and \(\Xn, n\in\nat,\) are defined on the same probability space \((\Omega,\cA,\PP)\), and that for all \(\omega\in\Omega\) the set of jump times of \((X_t^{(n)}(\omega))_{t\geq0}\) is a subset of the jump times of \((X_t(\omega))_{t\geq0}\).\\
The proof of Theorem \ref{thm_cpn} (iii) under \textbf{(c1)} implies that \(\Xn\) is a stationary process with invariant distribution \(\ppi\). Assume \(X_0=X_0^{(n)}\), and let us show that for all \(t\geq0\) it holds \(X^{(n)}_t\to X_t\) in law for \(n\to\infty\). This type of continuous dependence on the coefficients is well-known for the case when \(\phi\) is locally Lipschitz continuous and satisfies a linear growth condition, cf. \cite[Thm. IX.6.9]{jacod}. Unfortunately, these conditions are not necessarily fulfilled in our case, as we require merely piecewise Lipschitz continuity for \(\phi\). In the following, \(\mathbf{p}\in\cP\) denotes a partition w.r.t. which \(\phi\) is piecewise Lipschitz continuous.\\
Fix \(\omega\in\Omega\) and \(T>0\), denote \(a:=X_0(\omega)\), and set
\begin{align*}
t_1:=\inf\{t\geq0: \Delta X_t(\omega)\neq0~~ \text{ or } ~~X_t(\omega)\in\mathbf{p}\}.
\end{align*} 
Further, for \(t\geq0\) denote by \(q(t):=X_t(\omega)\) and \(q_n(t):=X^{(n)}_t(\omega)\) the paths of the respective processes. On the interval \([0,t_1)\), \(q\) and \(q_n\) are governed by the autonomous integral equations
\begin{align*}
q(t)=a+\int_0^{t_1}\phi(q(s))\diff s
\end{align*}
and
\begin{align*}
q_n(t)=a+\int_0^{t_1}\phi_n(q_n(s))\diff s,
\end{align*}
respectively. Hence, for all \(t\in[0,t_1)\)
\begin{align*}
|q(t)-q_n(t)| &\leq\int_0^{t_1} |\phi(q(s))-\phi_n(q_n(s))|\diff s\\
&\leq \int_0^{t_1}|\phi(q(s))-\phi(q_n(s))|+|\phi(q_n(s))-\phi_n(q_n(s))|\diff s\\
&\leq \int_0^{t_1} \ell|q(s)-q_n(s)| \diff s + c\mu((0,1/n]\cup[n,\infty)) 
\end{align*}
for some constants \(\ell,c>0\). This is due to the estimate in \eqref{eq_estimatephi}, the fact that \(q\) is strictly decreasing, and to the piecewise Lipschitz continuity of \(\phi\). For the latter we note that \(\phi\leq\phi_n\) which implies \(q(t)\leq q_n(t)\) for all \(t\in[0,t_1]\). In other words, if \(q\) reaches a discontinuity of \(\phi\) at \(t_1\), i.e. \(q(t_1)\in\mathbf{p}\), then it reaches it ahead of \(q_n\) which allows the estimate above. \\
Grönwall's inequality (cf. \cite[Cor. I.6.6]{hale}) then yields for all \(t\in[0,t_1)\)
\begin{align}\label{eq_conv1}
|q(t)-q_n(t)|\leq c\mu((0,1/n]\cup[n,\infty)) \left(1+\int_0^{t_1}\ee^{\ell(t-s)}\diff s\right)
\end{align}
which vanishes for \(n\to\infty\).\\
Our strategy is now to iterate this step until we surpass the time \(T\). By design there are two cases: \(q\) either jumps at \(t_1\) or hits a discontinuity of \(\phi\). Note that it can be ruled out that both events occur at the same time as the probability of this happening is zero. In the same way we exclude \(q\) jumping onto a discontinuity of \(\phi\) because \(q\) is strictly decreasing, jumps are space homogeneous, and the set of discontinuities of \(\phi\) has no accumulation points in \((0,\infty)\).\\
The first case, i.e. \(q\) jumps at \(t_1\), is simple. Clearly, there exists \(N\in\nat\) such that for all \(n>N\) it holds \(\Delta L_{t_1}=\Delta L_{t_1}^{(n)}\). Consequently, for all \(\varepsilon>0\) there exists \(N'\in\nat\) such that for all \(n>N'\) it holds \(|q(t_1)-q_n(t_1)|<\varepsilon\). Choosing \(\varepsilon\) small enough ensures that \(q\) and \(q_n\) both jump into the same interval \((x_i,x_{i+1})\) of the partition w.r.t. which \(\phi\) is piecewise Lipschitz continuous. Let
\begin{align}\label{def_t2}
t_2:=\inf\{t\geq t_1: \Delta X_t(\omega)\neq0~~ \text{ or }~~ X_t(\omega)\in\mathbf{p}\}.
\end{align}
For large enough \(n\in\nat\) we obtain 
\begin{align*}
|q(t)-q_n(t)|\leq \varepsilon + \int_{t_1}^{t_2} \ell|q(s)-q_n(s)| \diff s + c\mu((0,1/n]\cup[n,\infty)).
\end{align*}
for all \(t\in[t_1,t_2)\), and some (possibly different) constants \(\ell,c>0\). Applying Grönwall's inequality again concludes this iteration step.\\
For the second case, i.e. if \(q\) hits a discontinuity of \(\phi\) at \(t_1\), we argue differently. We denote by
\begin{align*}
t_2(n):=\inf\{t\geq t_1: q_n(t)=q(t_1)\}
\end{align*}
the time at which \(q_n\) also reaches the discontinuity of \(\phi\) at \(q(t_1)\). We require some observations: First, \(t_2(n)<\infty\) for all \(n\in\nat\) large enough since \(\phi_n<0\) is bounded away from zero on compact sets \(K\subset(0,\infty)\). Second, \(q_n(t_1)\to q(t_1)\) for \(n\to\infty\) due to \eqref{eq_conv1} and the fact that in this case \(q\) and \(q_n\) are continuous on \([0,t_1]\). Third, from \(\phi_n\leq \phi_{n-1}<0\) for all \(n\) large enough it follows that \(q_n'\leq q_{n-1}'<0\) on \([t_1,t_2(n)]\).\\
Thus, \(t_2(n)\to t_1\) for \(n\to\infty\). Further, also by the continuity of \(q\), it holds \(q(t_2(n))\to q(t_1)\) for \(n\to\infty\). Hence, for every \(\varepsilon>0\) we can choose \(N\in\nat\) such that for all \(n>N\)
\begin{align*}
|q(t_2(n))-q_n(t_2(n))|=|q(t_2(n))-q(t_1)|<\varepsilon.
\end{align*}
Using the above definition \eqref{def_t2} of \(t_2\) we obtain for all \(t\in[t_2(n),t_2)\)
\begin{align*}
|q(t)-q_n(t)|\leq \varepsilon' + \int_{t_2(n)}^{t_2} \ell|q(s)-q_n(s)| \diff s +  c\mu((0,1/n]\cup[n,\infty)).
\end{align*}
Note that on \([t_2(n),t_2)\) both \(q\) and \(q_n\) act on the same interval \((x_i,x_{i+1})\) of the partition \(\mathbf{p}\) w.r.t. which \(\phi\) is piecewise Lipschitz continuous. Grönwall's inequality then shows that for all \(t\in[t_2(n),t_2)\) it holds \(q_n(t)\to q(t)\) for \(n\to\infty\). But because \(q_n\) is continuous on \([0,t_2)\) and \(t_2(n)\to t_1\) for \(n\to\infty\), this property extends to \([t_1,t_2)\).\\
Finally, iteration and the fact that \(T\) and \(\omega\) have been chosen arbitrarily yields \(X^{(n)}_t\to X_t\) almost surely for \(n\to\infty\) and all \(t\geq0\). This implies weak convergence, and therefore, \(X_t\sim\ppi\) for all \(t\geq0\), i.e. \(\ppi\) is invariant for \(\X\). By Theorem \ref{thm_cpn} (i), \(\X\) is positive Harris recurrent, and hence, \(\ppi\) is the unique invariant distribution of \(\X\).
\end{proof}

\section{Outlook}\label{sec_out}
It is only natural to try extending Theorem \ref{thm_infinv} and Theorem \ref{thm_cpn} to more general settings, e.g. higher dimensions, more complicated driving noises, or target measures with disconnected supports, heavy tails or atoms. We believe that for many of these cases similar methods to the ones we employed here yield similar results. Generally speaking, one only has to show that the process in question is positive Harris recurrent, and that there exists a unique infinitesimally invariant distribution. The remaining steps are in most instances trivial or at least easy to prove under mild conditions.\\
In the following we shortly comment on the problems one faces when trying this approach on some of the more general cases. 

\paragraph*{Jump measures with heavy tails} The restriction to light tailed jumps, that is when \(\EE\xi_1<\infty\), is only needed in the proof of Theorem \ref{thm_cpn} (i). It is not entirely clear whether additional conditions are necessary to show positive Harris recurrence in the case of heavy tailed jumps, i.e \(\EE\xi_1=\infty\). Yet, as heavy tailed jumps imply larger (negative) values of the drift coefficient \(\phi\) by the definition in \eqref{eq_driftcpn} searching a more sophisticated norm-like function is the most promising approach.

\paragraph*{Subordinators as driving noise} In Assumption \textbf{(a1)} we required \(\L\) to be a spectrally positive compound Poisson process if \(\cE=(0,\infty)\). In Theorem \ref{thm_infinv} one might also wish to allow (pure jump) subordinators for \(\L\), i.e. Lévy processes with characteristic triplet \((0,0,\mu)\) for which \(\supp\mu\subset\real_+\) and \(\int_{(0,\infty)} (1\wedge z)\mu(\diff z)<\infty\).\\
However, in this case \eqref{eq_subor} in the proof of Theorem \ref{thm_infinv} does in general not hold. Thus, one needs to find a different way of showing that \(0\notin\cO\), e.g. by proving that \((\ln(X_t))_{t\geq0}\) does not explode. Moreover, uniqueness of the infinitesimally invariant distribution, that is Theorem \ref{thm_cpn} (iii), has to be shown differently. This is because there exists no subinterval of \((0,\infty)\) that cannot be reached by jumps, and because a non-zero amount of jumps does occur almost surely during every time interval. Thus, the proofs of Theorem \ref{thm_cpn} (iii) under \textbf{(c1)} and \textbf{(c2)}, respectively, do not apply in this case.

\paragraph*{Target distributions with full support} One might wish to extend the results of Theorem \ref{thm_cpn} for the case when \(\cE=\real\). However, just like in the previous paragraph, the approach used for the proof of Theorem \ref{thm_cpn} (iii) fails due to the fact that there exists no subinterval of \(\real\) that cannot be reached by jumps. Therefore, the proof of the uniqueness of the solution of \eqref{eq_iile} requires different arguments.

\paragraph*{Target measures with disconnected supports} Allowing only \(\cE=\real\) or \(\cE=(0,\infty)\) seems restrictive as the ability to cross gaps is one of the main advantages of the presence of jumps. Intending to allow disconnected supports of the target measure \(\ppi\) one has to assume three things: 
\begin{enumerate}
\item \(\cE\) is an open set,
\item jumps in both directions are possible, i.e. \(\Pi((-\infty,0))>0\) and \(\Pi((0,\infty))>0\),
\item jumps can only land in \(\cE\), i.e. \(\cE+\supp\Pi\subseteq\cE\).
\end{enumerate}
With those three assumptions one can show that, apart from \(\cE=\real\) and \(\cE\) being some half-line, the only option is that \(\cE\) is periodic, that is there exists \(p>0\) such that \(\cE+p=\cE\).\\
However, if \(\cE\neq\real\) is periodic and \(\X\) with state space \(\cO=\cE\) solves \eqref{eq_sde}, then \(\X\) cannot be positive Harris recurrent - regardless of the drift coefficient \(\phi\). The reason for this is simple: The jumps of \(\X\) are space-homogeneous, and \(\cE\) consists of countably many intervals of the same length that can only be connected by jumps. Thus, the mass of an invariant measure concentrated on each of these segments is the same. Therefore, no invariant measure can be finite (apart from the trivial measure).

\paragraph*{Target measures with atoms} An invariant measure \(\ppi\) with \(\ppi(\{x_0\})>0\) for one or more \(x_0\in\real\) can only be achieved by a solution \(\X\) of \eqref{eq_sde} if \(\X\) comes to a halt at \(x_0\). One possible solution might be to set \(\phi(x_0)=0\). At least heuristically this makes sense considering that the denominator in the original definition \eqref{eq_drift_coeff} of \(\phi\) is the density function of \(\ppi\). \\
However, extending Theorem \ref{thm_cpn} to this case needs a new idea since the current proof relies on the fact that any solution of \eqref{eq_iile} can be associated to a locally integrable function.

\paragraph*{Target measures with arbitrary tails} By Assumption \textbf{(b1)}, we require \(\ppi\) to have an exponential tail. This is mostly needed in the proof of Theorem \ref{thm_cpn} (i). As with heavy tailed jumps, using a more sophisticated norm-like function will most likely enable us to consider target measures \(\ppi\) for which only \(|\pi(x)|\leq c\ee^{-\alpha x}\) for all \(x\gg1\) and some constants \(c,\alpha>0\). \\
In case \(\ppi\) has a heavy tail, that is when \(|\pi(x)|\geq c x^{-(1+\alpha)}\) for all \(x\gg1\) and some constants \(c,\alpha>0\), it is not clear whether \(\X\) is positive Harris recurrent or not.

\paragraph*{Higher dimensions} Theorem \ref{thm_infinv} can be extended easily to target measures \(\ppi\) on \((\real^d,\cB(\real^d))\) and \(d\)-dimensional driving noises \(\L\) with \(d\geq2\). Simply use the multi-dimensional counterparts (cf. \cite{behmeoechsler}) to all occurring terms in the definition \eqref{eq_drift_coeff} of the drift coefficient, and make sure that jumps can only land in \(\cE\), and that \(\X\) cannot drift onto \(\partial \cE\).\\
However, Equation \eqref{eq_iile} becomes a partial differential equation in the multi-dimensional case. Thus, Theorem \ref{thm_cpn} cannot be extended with the same approach.

\paragraph*{Lévy-type driving noise} A solution to some of the problems mentioned above, e.g. disconnected supports or atoms, might be to select space dependent driving noises. For the case when \(\L\) is a Lévy-type process (for details see \cite{bottcher}) \cite{behmeoechsler} provides the required framework for defining the drift coefficient. Just like with higher dimensions extending Theorem \ref{thm_infinv} to this setting is feasible while the extension of Theorem \ref{thm_cpn} might require a new approach.

\section{Acknowledgments}\label{sec_ack}
I would like to thank Anita Behme for her advice and I am very grateful for her helpful comments and suggestions.

\bibliographystyle{plain}

\nocite{protter}
\nocite{behmeoechsler}
\nocite{MTI}
\nocite{MTII}
\nocite{MTIII}
\nocite{eliazar}

\bibliography{localLAW}

\end{document}